\documentclass[]{article}
\pdfoutput=1
\usepackage{microtype}
\usepackage{graphicx}
\usepackage{subfigure}
\usepackage{booktabs, multirow} 
\usepackage[para]{threeparttablex} 
\usepackage[table]{xcolor}
\usepackage{makecell}
\usepackage{cellspace}
\usepackage{colortbl}
\usepackage{multicol}
\usepackage{hyperref}
\hypersetup{
	colorlinks = false,
    hidelinks,
}

\usepackage{fancyhdr}
\usepackage{color}
\usepackage{algorithm}
\usepackage{algorithmic}
\usepackage{natbib}
\usepackage{forloop}
\usepackage[left=2.5cm,top=3cm,right=2.5cm,bottom=3cm]{geometry}
\usepackage[T1]{fontenc}

\usepackage[OT1]{fontenc}

\usepackage[american]{babel}

\usepackage[utf8]{inputenc}

\usepackage{amsmath,amssymb,amsfonts,mathrsfs, amsthm}

\usepackage{graphicx}

\usepackage{varioref}

\usepackage{mathtools}

\usepackage{array}

\usepackage{bm}

\usepackage{isomath}

\usepackage[shortlabels]{enumitem}

\newcommand{\N}{\mathbb{N}}

\newcommand{\R}{\mathbb{R}}

\newcommand{\vect}[1]{\mathbf{#1}}

\newcommand{\x}{\mathbf{x}}
\newcommand{\y}{\mathbf{y}}
\newcommand{\z}{\mathbf{z}}
\renewcommand{\u}{\mathbf{u}}
\newcommand{\bb}{\mathbf{b}}
\newcommand{\cc}{\mathbf{c}}
\newcommand{\vvv}{\mathbf{v}}
\renewcommand{\v}{\mathbf{v}} %

\newcommand{\B}{\mathbf{B}}
\newcommand{\C}{\mathbf{C}}
\newcommand{\U}{\mathbf{U}}
\newcommand{\Z}{\mathbf{Z}}
\newcommand{\I}{\mathbf{I}}
\renewcommand{\P}{\mathbf{P}}
\newcommand{\V}{\mathbf{V}}

\newcommand{\E}{\mathbb{E}}
\renewcommand{\Pr}{\mathbb{P}}

\newcommand{\ih}{\hat{i}}

\newcommand{\upp}{\mathcal{O}}
\newcommand{\low}{\Omega} 
\newcommand{\algo}{\mathsf{A}}
\newcommand{\oracle}{\mathsf{O}}
\newcommand{\ortho}{\mathsf{Ortho}}

\DeclarePairedDelimiter\abs{\lvert}{\rvert}
\DeclarePairedDelimiter\norm{\lVert}{\rVert}

\newcommand{\inner}[2]{ \langle {#1, #2} \rangle} 

\renewcommand{\epsilon}{\ensuremath\varepsilon}

\renewcommand{\phi}{\ensuremath{\varphi}}

\newcommand{\curir}{\x_t^s} 
\newcommand{\estim}{\v_t^s} 
\newcommand{\hestim}{\U_t^s} 
\newcommand{\snap}{\widehat{\x}^s} 
\newcommand{\snapgrad}{\mathbf{g}^s}
\newcommand{\snaphess}{\mathbf{H}^s}
\newcommand{\step}{\mathbf{h}_t^s}

\newcommand{\grdf}{\nabla F} 
\newcommand{\hesf}{\nabla^2 F} 
\newcommand{\grdfc}{\grdf(\curir)} 
\newcommand{\hesfc}{\hesf(\curir)} 
\newcommand{\grdfi}{\nabla f_{i_t}} 
\newcommand{\hesfi}{\nabla^2 f_{i_t}} 
\newcommand{\grdfj}{\nabla f_{j_t}} 
\newcommand{\hesfj}{\nabla^2 f_{j_t}} 

\newcommand{\xdiff}{\curir - \snap} 

\usepackage{dsfont}

\makeatletter
\def\thmheadbrackets#1#2#3{%
	\thmname{#1}\thmnumber{\@ifnotempty{#1}{ }\@upn{#2}}%
	\thmnote{ {\the\thm@notefont[#3]}}}
\makeatother

\newtheoremstyle{brakets}
{}
{}
{\itshape}
{}
{\bfseries}
{.}
{ }
{\thmheadbrackets{#1}{#2}{#3}}

\theoremstyle{brakets}

\newtheorem{theorem}{Theorem}[section]

\newtheorem{definition}[theorem]{Definition}
\newtheorem{lemma}[theorem]{Lemma}

\newtheorem{assumption}[theorem]{Assumption}

\renewenvironment{proof}[1]{{\bf Proof }[#1]{\bf.}}{\hfill $\Box$ \\}

\renewcommand{\cite}{\citep}
\title{On the Oracle Complexity of Higher-Order Smooth Non-Convex Finite-Sum Optimization}
\author{Nicolas Emmenegger\thanks{Corresponding author} \and Rasmus Kyng \and  Ahad N. Zehmakan}
\date{\texttt{nicolaem@ethz.ch  \{kyng,abdolahad.noori\}@inf.ethz.ch} }

\begin{document}

\maketitle

\begin{abstract}
	\noindent
We prove lower bounds for higher-order methods in smooth non-convex finite-sum optimization. Our contribution is threefold: We first show that a deterministic algorithm cannot profit from the finite-sum structure of the objective, and that simulating a $p$th-order regularized method on the whole function by constructing exact gradient information is optimal up to constant factors. We further show lower bounds for randomized algorithms and compare them with the best known upper bounds. To address some gaps between the bounds, we propose a new second-order smoothness assumption that can be seen as an analogue of the first-order mean-squared smoothness assumption. We prove that it is sufficient to ensure state-of-the-art convergence guarantees, while allowing for a sharper lower bound.
\end{abstract}
\section{Introduction}

Many problems in machine learning can be formulated as empirical risk minimization, viewing the loss of each data point as a component in a sum. This yields an objective function $F : \R ^ d \rightarrow \R$, $F(\x) = \frac{1}{n}\sum_{i=1}^n f_i(\x)$ that one minimizes under a variety of smoothness assumptions. The ultimate goal would be to find
$$
\x^* = \arg \min_{\x \in \R^d} F(\x).
$$

Since finding such a global minimum is in general NP-Complete \cite{murty:nphardness}, theoretical guarantees are expressed in terms of weaker requirements. Inspired by necessary conditions for minima, customary guarantees are approximate first-order or second-order stationary points (FOSP, SOSP). We will focus here on the oracle complexity of finding an $\varepsilon$-approximate first-order stationary point of $F$, that is a point $\x$, such that $\norm{\nabla F(\x)} \leq \epsilon$, which is standard for lower bounds in non-convex optimization \cite{carmon:lower:i, carmon:lower:ii, carmon:lower:stoc, fang:spider, gu:lower}.

 When data sets are large, gradients are often approximated by evaluating only a subset of all training examples \cite{bottou:review}. This leads to a model where in each iteration of an algorithm, one component function $f_i$'s derivative information can be queried. In this model, the most prevalent algorithms today are stochastic gradient descent (SGD) and variants thereof. However, more query efficient algorithms have been explored. Variance reduction techniques -- first introduced in convex optimization \cite{johnsonzhang:svrg_convex} -- have been successfully applied in the non-convex setting: see e.g. \citet{allen:nc_vr}, \citet{reddi:nc_vr} or \citet{lei:nc_vr} for early works. These algorithms draw their speedup from cleverly constructed low-variance gradient estimators.

The best known rate for gradient-based algorithms has first been achieved by the SPIDER algorithm developed by \citet{fang:spider}. Under the assumption that the component functions are mean-squared smooth, their algorithm finds a FOSP in $\upp(\sqrt{n}\epsilon^{-2})$ first-order oracle calls.
Subsequent work has not improved on this convergence rate, but tried to improve practicality, see e.g.~\citet{wang:spiderboost}.

\subsection{Higher-order variance-reduced methods}
\label{sec:higherorder_algo}
Motivated by the fact that higher-order algorithms can give guarantees in terms of SOSPs and typically enjoy better convergence rates in a non-finite-sum, noiseless setting \cite{nesterov:cr}, there have been successful attempts to apply variance reduction techniques to higher-order algorithms.

While there exist approaches exploiting third-order derivatives \cite{lucchi:tensor}, most work has focused on gradient and Hessian based algorithms. The first to use inexact Hessian information while retaining global convergence in the non-convex finite-sum setting are \citet{kohler:subsampled}, by using a sub-sampled Hessian approximation scheme. However, with decreasing step-sizes, their sample sizes may approach $n$.

Subsequent work has improved the dependence on $n$: \citet{zhou:svrc} give a method (SVRC) that uses only $\tilde{\upp}(n^{4/5}\epsilon^{-3/2})$ \footnote{We use $\tilde{\upp}$ to hide polylogarithmic factors in $d$, $n$ and $1/\epsilon$} second-order oracle queries to find a SOSP under a second-order smoothness assumption on each of the $f_i$'s. This method relies on semi-stochastic gradient and Hessian estimators inspired by first-order variance-reduction techniques. \citet{shen:trust_region} provide an even faster trust-region method (STR2) that achieves the second-order oracle complexity of $\tilde{\upp}(n^{3/4}\varepsilon^{-3/2})$, but under the stronger assumption that the gradient is Lipschitz continuous as well (i.e. first and second-order smoothness). 
  
There is a line of research which tries to minimize Hessian complexity at the cost of additional gradient queries: \citet{shen:trust_region} also give the Algorithm STR1 that finds a SOSP in $\tilde{\upp}(\min(\sqrt{n}\epsilon^{-2} , {n}\epsilon^{-3/2}))$ gradient accesses and $\tilde{\upp}(\min(\sqrt{n}\epsilon^{-3/2}, \epsilon^{-2}))$ Hessian accesses. \citet{zhou:newest_stochastic} provide a method that solves the same problem with $\tilde{\upp}(\min(\sqrt{n}\epsilon^{-2} , {n}\epsilon^{-3/2}, \epsilon^{-3}))$ gradient accesses and $\tilde{\upp}(\min(\sqrt{n}\epsilon^{-3/2}, \epsilon^{-2}))$ Hessian queries.
  
For the higher-order oracle complexity measure that we will focus on here, SVRC and STR2 represent the best known upper bounds for second-order randomized algorithms. As we only assume $p$th-order smoothness, we will take SVRC \cite{zhou:svrc} as reference for second-order methods.

\renewcommand{\TPTminimum}{\linewidth}
\newcommand{\mytnote}[1]{\tnote{\textnormal{#1}}}
\begin{table}[t]
\label{tab:individual}
\caption{A comprehensive overview of the upper and lower bounds for incremental first-, second- and higher-order oracle models. $p$ refers to the degree of smoothness of the function(s), and $n$ to the number of components in the finite-sum structure. These bounds assume that each function $f_i$ is $p$th-order smooth, i.e. has Lipschitz $p$th-order derivative tensor. The first row refers to deterministic algorithms while the three below concern the randomized setting. Our contributions are highlighted in grey.}
\begin{center}
\begin{threeparttable}
\vskip -0.06in
\begin{small}
\begin{sc}
\begin{tabular}{SlSlScSc}
\Xhline{2\arrayrulewidth}
& &Upper bound & Lower bound  \\
\hline
Deterministic &
 & $\upp(n\epsilon^{-\frac{p+1}{p}})$ \mytnote{a}
 & \cellcolor[HTML]{C0C0C0} $\low(n \epsilon^{-\frac{p+1}{p}})$ \mytnote{b} \\
 \hline
\multirow{5}{*}{Randomized} &
 $p=1$ 
 &$\upp(n^{\frac{1}{2}}\epsilon^{-2})$ \mytnote{c}
 & $\low(\epsilon^{-2})$ \mytnote{d;f} \\
  \cline{2-4}
 &$p = 2$
 & $\tilde{\upp}(n^{\frac{4}{5}}\epsilon^{-\frac{3}{2}})$ \mytnote{e}
 & \cellcolor[HTML]{C0C0C0} $\low(n^{\frac{1}{4}}\epsilon^{-\frac{3}{2}})$ \mytnote{f} \\
 \cline{2-4}
 &$p > 2$
 & $\upp(n\epsilon^{-\frac{p+1}{p}})$ \mytnote{a}
 & \cellcolor[HTML]{C0C0C0} $\low(n^{\frac{p-1}{2p}}\epsilon^{-\frac{p+1}{p}})$ \mytnote{f} \\
\Xhline{2\arrayrulewidth}
\end{tabular}
\end{sc}
\end{small}
\end{threeparttable}
\end{center}
\par\medskip
\vskip 0.05in
\footnoterule
\begin{tablenotes}
{\footnotesize
\item [a] \citet{birgin:regularized}
\item [b] Theorem~\ref{theorem:deterministic}
\item [c] \citet{fang:spider}
\item [d] \citet{gu:lower}
\item [e] \citet{zhou:svrc}
\item [f] Theorem \ref{theorem:individual_smoothness}}
\end{tablenotes}
\vskip -0.1in
\end{table}

\subsection{Related work on lower bounds}
Lower bounds for smooth non-convex optimization have all built on the works of \citet{carmon:lower:i,carmon:lower:ii}. These papers focus on the case where the objective is composed of a single smooth function (i.e., $n=1$) and full gradient information is available at each iteration. In the first paper they establish the optimal rate of $\Theta(\varepsilon^{-(p+1)/p})$ to find $\epsilon$-approximate FOSPs for algorithms having access to as much derivative information as needed under the assumption that the function is $p$th-order smooth. In the companion paper, they provide lower bounds for first-order algorithms.

\renewcommand{\TPTminimum}{\linewidth}
\begin{table}[t]
\label{tab:alternative}
\caption{Lower and upper bounds for randomized algorithms under alternative smoothness assumptions. For the case $p=1$, we assume mean-squared smoothness and for the case $p=2$ we assume our new third-moment smoothness assumption (Assumption \ref{assumption:third}). Our contributions are again highlighted in grey.}
\begin{center}
\begin{threeparttable}
\vskip -0.06in
\begin{small}
\begin{sc}
\begin{tabular}{SlScSc}
\Xhline{2\arrayrulewidth}
&Upper bound & Lower bound  \\
\hline
 $p = 1$
 & $\upp(n^{\frac{1}{2}}\epsilon^{-2})$ \mytnote{a}
 & $\low(n^{\frac{1}{2}}\epsilon^{-2})$ \mytnote{b} \\
 \hline
 $p = 2$
 & \cellcolor[HTML]{C0C0C0} $\tilde{\upp}(n^{\frac{4}{5}}\epsilon^{-\frac{3}{2}})$ \mytnote{c}
 & \cellcolor[HTML]{C0C0C0} $\low(n^{\frac{5}{12}}\epsilon^{-\frac{3}{2}})$ \mytnote{d} \\
\Xhline{2\arrayrulewidth}
\end{tabular}
\end{sc}
\end{small}
\end{threeparttable}
\end{center}
\par\medskip
\vskip 0.05in
\footnoterule
\begin{tablenotes}
{\footnotesize
\item [a] \citet{fang:spider}
\item [b] \citet{fang:spider, gu:lower}
\item [c] Theorem~\ref{theorem:third_moment:upper}
\item [d] Theorem~\ref{theorem:third_moment:lower}}
\end{tablenotes}
\vskip -0.1in
\end{table}

In the same paper where \citet{fang:spider} introduce the first-order algorithm SPIDER with $\upp(n^{1/2}\epsilon^{-2})$ gradient oracle complexity, they also show their algorithm to be optimal, up to constant factors, for the mean-squared smooth finite-sum setting.

Furthermore, \citet{gu:lower} prove lower bounds on first-order algorithms for a variety of regimes in finite-sum optimization, including the non-convex case. A shortcoming of their results is that they place a linear-span restriction on the algorithms in question, i.e. the iterates of considered algorithms stay in the span of the queried gradients.

It is also worth noting that \citet{carmon:lower:stoc} and \citet{arjevani:lower:stoc:second} prove lower bounds for a related but different stochastic (online) setting. In this model one does not assume a finite-sum structure, but typically places variance assumptions on the queried gradients. The first paper focuses on first-order stationary points, while the second is considering approximate local minima and higher-order algorithms. We will not further study this setting here.

\subsection{Our contribution}
We give the first lower bound results for the problem of finding an approximate stationary point of a sum of $p$th-order individually smooth non-convex functions, in a model where an algorithm queries the derivatives of individual functions at each time-step. We provide lower bounds for both deterministic and randomized algorithms. An overview is given in Table 1.

First we consider deterministic algorithms and show that a $p$th-order regularized method that constructs the full derivative at each iteration is optimal up to constant factors.
We use an adversarial construction that forces the algorithm to spend a large number of queries to discover useful information. To the best of our knowledge, this result is also new for the widely studied case of first-order smooth non-convex finite-sum optimization and implies that gradient descent on the full function is optimal up to constant factors. The result demonstrates a clear separation between deterministic and randomized algorithms. 

Further, we give the first lower bounds for randomized algorithms in this setting, which allow comparison with a new line of research of higher-order variance reduction. 
We derive the bounds with a probabilistic construction, building on the family of zero-chain functions first introduced by \citet{carmon:lower:i}. In contrast to the first-order case studied by \citet{gu:lower}, we show a non-trivial dependence on $n$ for the $p>1$ regime. 

There is a gap between the best known upper bounds and our lower bound under the individual smoothness assumption.
To alleviate this gap, we introduce a new, weaker notion of second-order smoothness and show that it is sufficient to guarantee state-of-the-art oracle complexities for second-order variance-reduced methods, while allowing for a tighter lower bound. Table~2 shows our bounds and contrasts them with analogous results using mean-squared smoothness in the first-order setting. To upper bound the oracle complexity, we show that the variance of SVRC's \cite{zhou:svrc} Hessian and gradient estimators can be controlled via the second-order mean-cubed smoothness of the finite-sum function. 

All our bounds are tight in terms of $\epsilon$ dependence, but closing the gaps with respect to the dependence on $n$ remains an interesting open problem.

\section{Model and assumptions}
In this section, we will introduce the model we work in for deriving our lower bounds.
\subsection{Problem description}
As mentioned above, we focus on finding $\epsilon$-approximate first-order stationary points. We assume that a problem instance is a function $F = \frac{1}{n}\sum_{i=1}^n f_i$, which satisfies the following assumption.
\begin{assumption}
\label{assumption:individual}
We say $F \in \mathcal{F}_p^n(\Delta, L_p)$ if for some $d$, $F : \R^d \rightarrow \R, \, \x \mapsto \frac{1}{n}\sum_{i=1}^n f_i(\x) $ satisfies the following properties
\begin{enumerate}[i)]
\item Each function $f_i$ is $p$th-order smooth, i.e. it is $p$ times continuously differentiable, and for all $\x,\y$ \footnote{$\norm{\cdot}$ always refers to the tensor operator norm, e.g. to the euclidean norm for vectors and the spectral norm for matrices}
$$
	\norm{\nabla^{p}f_i(\x) - \nabla^{p}f_i(\y)} \leq L_p \norm{\x-\y}.
$$
\item Assuming that an algorithm starts at iterate $\x_0 = \mathbf{0}$, the initial gap to optimality is bounded by
$$\frac{1}{n}\sum_{i=1}^n f_i(\x_0) - \inf_\x \frac{1}{n}\sum_{i=1}^nf_i(\x) \leq \Delta.
$$
\end{enumerate}
\end{assumption}
Whenever $n$, $p$, $L_p$ and $\Delta$ are obvious from context, we say that $F$ satisfies Assumption~\ref{assumption:individual} if $F \in \mathcal{F}_p^n(\Delta, L_p)$.
Furthermore, note that if the function is $p+1$ times differentiable, then the first property is equivalent to requiring $\norm{\nabla^{p+1} f_i(\x)} \leq L_p$. 
\subsection{Algorithm and oracle models}
Usually, when $p$th-order smoothness is assumed, one works with derivatives up to the $p$th order. Therefore, in the interest of deriving lower bounds, it is even stronger to let the algorithm have access to as many derivatives as it would require. It turns out that this actually will not change the bounds, and they depend only on the order of smoothness $p$ of the considered function. We assume that an algorithm queries iterates according to the following definition, and we will lower bound the number of such queries it needs to do to reach its objective.
\begin{assumption}
	In the incremental higher-order oracle model (IHO), an oracle for a function $F = \frac{1}{n} \sum f_i$ consists of a mapping\footnote{We write $i:j$ or $[i:j]$ for the set of integers $\{i,\ldots,j\}$ and let $[m] := [1:m]$. Furthermore, we define $\R^{\otimes^k d}$ to be the space of $k$-dimensional tensors over $\R^d$. We denote by $\nabla^{(0:q)}$ the union of derivative tensors up to the order $q$.}
	\begin{align*} 
		\oracle_F^{(q)} : \N \times \R^d &\rightarrow \left(\R, \R^d,...,\R^{\otimes^q d}\right) \\
		(i, \x) &\mapsto \nabla^{(0:q)}f_i(\x).
	\end{align*}
	We condense the notation by letting $\oracle_F^{(q)}(i^{0:t-1}, \x^{(0:t-1)})$ correspond to the union of all oracle responses before iteration $t$.
\end{assumption}

We can then think of an algorithm as generating a sequence of indices and iterates, namely those it queries the IHO on. 

\begin{assumption}
\label{assumption:algo}
We will assume that an algorithm $\algo$ has access to an infinite sequence of random bits $\xi \sim \mathcal{U}([0,1])$ drawn at the beginning of the procedure. \footnote{For a deterministic algorithm, we simply assume the sequence is fixed.} Then, $\algo$ consists of a sequence of mappings $\{A^{(t)}\}_{t \in \N}$ which produce indices and iterates based on previous oracle responses:
	\begin{align*}
		[i^t, \x^{(t)}]
		&= A^{(t)} 
		\Big \{
		\xi, i^{0:t-1}, 
		\x^{(0:t-1)}, 
		\oracle_F^{(q)}(i^{0:t-1}, \x^{(0:t-1)})
		\Big\}.
	\end{align*}
	Without loss of generality, we set $\x^{(0)} = \vect{0}$ because if a function $f$ is difficult to optimize for starting point $\vect{0}$, then $\x \mapsto f(\x - \x^{(0)})$ is difficult to optimize for starting point $\x^{(0)}$. Finally, we set no restrictions on how $i^0$ is chosen.
\end{assumption}
Note that this is a quite general assumption, merely capturing the fact that the algorithm performs ``something'' between different queries. Also note that in the finite-sum setting, any potential randomness is inside the algorithm and not the oracle.	
\subsection{Complexity measure}
Finally, we need a proper measurement to characterize the complexity of an algorithm. We choose the following.
\begin{definition}
We define the oracle complexity $T_\epsilon(\algo, F)$ of an algorithm $\algo$ on $F$ as the infimum over all $t \in \N$ such that the following holds with probability at most $\frac{1}{2}$
$$
    \forall s \leq t \, : \,  \norm{\nabla F(\x^{(s)})} > \epsilon.
$$
In other words, this corresponds to $t$ such that for all larger $t'$, with probability $1/2$ the algorithm will encounter an iterate $s\leq t'$ with sufficiently small gradient.
\end{definition}
We note that $T_\epsilon(\algo, F) \geq t$ implies that for all $s \leq t$, $P(\norm{\nabla F(\x^{(s)})} > \epsilon) \geq 1/2$, and so by Markov's inequality, $\epsilon / 2 \leq \epsilon P(\norm{\nabla F(\x^{(s)})} > \epsilon) \leq \E\norm{\nabla F(\x^{(s)})}$ for all $s \leq t$.
This implies that we can also compare our lower bounds to the methods which give guarantees in terms of a complexity that ensures an output with a small gradient \emph{in expectation}.
\newcommand\ddelta{\bm{\delta}}

\section{Lower bounds for deterministic algorithms}
\label{section:deterministic_bound}

In this section, we show that any algorithm that can not resort to randomness can outperform only by a constant factor one that simulates a higher-order regularized method \cite{birgin:regularized}. By the latter, we mean a procedure which constructs the full derivative information at each step by querying all $n$ functions.

Inspired by \citet{carmon:lower:i} and \citet{woodworth:tight}, we define a family of hard instances that we will later instantiate depending on the algorithm's behaviour. The main intuition is to utilize an underlying function which has a large gradient as long as there are coordinates left which are very close to zero. Depending on the queries of the algorithm, we will be able to adversarially and incrementally choose a rotation of the input space in such a way that these coordinates indeed stay close to zero for a long time. 
\begin{definition}
\label{def:deterministic:hard_familly}
Let $K \in \N$ and for $k \in [K]$ let $\delta_k \in \{0,1\}$ be arbitrary. We define the function $f_{K,\ddelta} : \R^{K} \rightarrow \R$ as
\begin{align*}
f_{K,\delta}(\x) &:= -\delta_1 \Psi(1)\Phi(x_1) \\ 
& + \sum_{k=2}^{K} \delta_k\left[\Psi(-x_{k-1})\Phi(-x_k) -\Psi(x_{k-1})\Phi(x_k)\right],
\end{align*}
where the functions $\Phi$ and $\Psi$ are given by
	$$
		\Psi(x) := \begin{cases}
		 0 & x \leq 1/2 \\
		 \exp\left(1- \frac{1}{(2x-1)^2}\right) & \text{otherwise}
		\end{cases}
	$$
	and
	$$
		\Phi(x) = \sqrt{e}\int_{-\infty}^{x}e^{-\frac{1}{2}t^2}\mathrm{d}t.
	$$
\end{definition}

We should emphasize that the function $\bar{f}_{K}$ defined by \citet{carmon:lower:i} can be represented by $f_{K, \mathbf{1}}$.

For the remaining parts of this section, assume the algorithm $\algo$, the number of functions $n$ and parameters $\Delta$ and $L_p$ to be fixed. The idea is to construct $n$ functions of the above family, where each function $f_i(\x)$ will be given (modulo rescaling) by $f_{K+1,\ddelta_i}(\V^T \x)$ for some suitable $\ddelta_i \in \{0,1\}^{K+1}$ and shared $\V \in \R^{d\times K+1}$. In the convex finite-sum setting, an analogous construction is exploited by \citet{woodworth:tight} for first-order algorithms. We will split up the iterates of the algorithm in rounds, starting at $k=2$ and ending at $k=K+1$. Thus after round $k$, in total $k-1$ rounds will have elapsed. We define a round to span queries to $\lceil{n/2}\rceil$ \emph{different} functions.
With those concepts in hand, we define the hard instance as:
\begin{definition}
\label{definition:deterministic:hardinstance}
For $i \in [n]$ let $\delta_{i,1} = \mathbf{1}[i \leq \lceil n/2 \rceil]$. For $k \in [2:K+1]$ let $\delta_{i,k} = 1$ iff $\algo$ does not query function $i$ during round $k$. Further, let $d \geq K+1$ and let $\V \in \ortho(d,K+1)$ be a matrix with orthonormal columns. Let $\lambda, \sigma > 0$ be parameters we will fix later. Then, we define
$$
    f_i(\x) = \lambda \sigma^{p+1} f_{K+1,\ddelta_i}\left(\V^T\x / \sigma\right),
$$
and consequently $F(\x)=\frac{1}{n}\sum_{i=1}^n f_i(\x)$.
\end{definition}
We now prove that there exists an adversarial rotation with the following property:
\begin{lemma}
\label{lemma:deterministic:smallprod}
In Definition~\ref{definition:deterministic:hardinstance}, $\V$ can be chosen such that for the sequence of indices and iterates $\{[i^t, \x^{(t)}]\}$ that algorithm $\algo$ produces up to the end of round $K+1$, we have $\inner{\v_{K+1}}{\x^{(t)}} = 0$ for all $t$.
\end{lemma}

\begin{proof}{of Lemma \ref{lemma:deterministic:smallprod}}
We will omit the scaling parameters as they do not influence the proof in any way and define for $k \in [K]$ the shorthand $y_k = y_k(\x) =  \inner{\v_k}{\x}$.
We will construct the oracle such that during round $r \in [2:K+1]$, its responses are based on the function:
\begin{align*}
f_i^{r}(\x) &= -\delta_{i,1} \Psi(1)\Phi(y_1) \\ & + \sum_{k=2}^{r-1} \delta_{i,k}\left[\Psi(-y_{k-1})\Phi(-y_k) -\Psi(y_{k-1})\Phi(y_k)\right].
\end{align*}
We will show that $\V$ can be chosen such that these responses are consistent with Definition~\ref{definition:deterministic:hardinstance}. By consistence, we mean equality of the function values and derivatives at the queried indices and points.

By construction, the answers for round $r$, only depend on $\v_k$ and $\delta_{i,k}$ for $k < r$. This allows us to determine $\delta_{i,r}$ and $\v_r$ at the \emph{end} of round $r$. Specifically, we will choose $\v_r$ such that $\inner{\v_r}{\x^{(t)}} = 0$ for all iterates occurring before the end of round $r$ (i.e. all queries made so far). Further, $\v_r$ needs to be orthogonal to $\v_k$ for all $k < r$. These orthogonality constraints imply a requirement on the dimension of the domain of $F$. This dimension $d$ must therefore be linear in the sum of $K$ and of the final lower bound, to ensure orthogonality to both iterates and between the columns of $\V$ is possible. As mentioned above, we will also choose $\delta_{i,r} = 1$ iff function $i$ was not queried during round $r$. 

We must now prove that for all $q \geq 0$ and iterates $t$ queried during round $r$, we have  $\nabla^q{f_{i^t}^r}(\x^{(t)}) = \nabla^q{f_{i^t}}(\x^{(t)})$, guaranteeing that our oracle is aligned with the function from Definition~\ref{definition:deterministic:hardinstance}. For simplicity, we define $\x = \x^{(t)}$ and $i = i^t$. Then, we can write $f_i(\x)$ as
\begin{align*}
    f_i^r(\x) + \delta_{i,r}[\Psi(-y_{r-1})\Phi(-y_r) -\Psi(y_{r-1})\Phi(y_r)] + g_i^r(\x)
\end{align*}
for $g_i^r(\x) = f_i(\x) - f_i^r(\x)$. Since function $i$ was queried during round $r$, we have $\delta_{i,r} = 0$, and so $f_i(\x) = f_i^r(\x) + g_i^r(\x)$. Hence, it suffices that $\nabla^q g_i^r(\x) = \mathbf{0} \in \R^{\otimes^q d}$. Indeed, $\Psi(z) = 0$ for all $\abs{z} \leq 1/2$. By our choice of $\V$, we have $\inner{\v_k}{\x} = 0$ for all $k \geq r$. Since all terms in $g_i^r$ have a multiplicative factor $\Psi(\pm \inner{\v_{k-1}}{\x})$ for some $k \geq r+1$, the function $g_i^r$ is indeed constant 0 inside a neighbourhood of $\x$, and so all its derivative tensors are $\mathbf{0}$ at $\x$.
\end{proof}
We should stress that a key property of the function $\bar{f}_{K} = f_{K,\mathbf{1}}$ is that as long as the last coordinate in its input is zero, the gradient of the function will be lower bounded by a constant. 
\begin{lemma}[Lemma 2 in \citet{carmon:lower:i}]
	\label{lemma:original:largegradient}
	Let $\x \in \R^K$ with $\abs{x_k} < 1$ for some $k \in [K]$. Then, there exists $l \leq k$ with $\abs{x_l} < 1$ and
	$$
		\abs*{\frac{\partial \bar{f}_K}{\partial x_l }(\x)}  > 1.
	$$
\end{lemma}
This property can be transferred to $F = \frac{1}{n}\sum f_i$:
\begin{lemma} For all iterates up to the end of round $K+1$, we have $\left(\V^T\x^{(t)}\right)_{K+1} = 0$, and so
$$
    \norm{\nabla F(\x^{(t)})} > \frac{\lambda \sigma^p}{4}.
$$
\end{lemma}

To show the main result, we merely have to set the scaling parameters such that our function respects Assumption \ref{assumption:individual}. Note that $\lambda$ controls the smoothness parameter, $\sigma$ controls the gradient norm lower bound and $K$ needs to be chosen as large as possible, but in a way that makes $F$ respect the initial optimality gap $\Delta$. Together, they can be chosen to imply the theorem below. 

\begin{theorem}
\label{theorem:deterministic}
For any $p$ and deterministic algorithm $\algo$ satisfying Assumption~\ref{assumption:algo}, for any $n$, $\Delta$ and $L_p$ and $\varepsilon$ there exists a function $F \in \mathcal{F}_p^n(\Delta, L_p)$ such that
\begin{equation}
 T_\epsilon(\algo, F) \geq \Omega \left( \left(\frac{L_p}{\ell_p}\right)^{1/p}\frac{\Delta n}{\varepsilon^{(p+1)/p}}\right), \label{eq:deterministic:omega}
\end{equation}
where the constant factors hidden by $\Omega$ do not depend on $n$, $\varepsilon$ or $p$ and $\ell_p \leq \exp{\left(\frac{5}{2}p\log p + cp\right)}$ for some constant $c < \infty$. Moreover, the dimension of this function merely needs to be of the same order as \eqref{eq:deterministic:omega}.
\end{theorem}

To summarize, we get a lower bound of $\Omega\left(n\varepsilon^{-(p+1)/p}\right)$. In the noiseless $n=1$ setting, the optimal complexity is characterized by $\Theta(\varepsilon^{-(p+1)/p})$ \cite{carmon:lower:i}. Indeed, \citet{birgin:regularized} prove this to be achievable with higher-order regularized methods, subsuming results known for gradient descent and cubic regularization. These methods also imply that Theorem~\ref{theorem:deterministic} characterizes the optimal oracle complexity, as we can simulate a higher-order regularized method by spending $n$ queries at each iterate.
\section{Lower bounds for randomized algorithms}

When constructing hard instances for randomized algorithms, one does not have the luxury of reacting to the algorithms queries, because we cannot anticipate the random seed $\xi$. The approach taken here, and in prior work, is to draw orthogonal vectors from some high-dimensional space and show that with some fixed probability, the iterates of the algorithm will be close to orthogonal to this set of important directions. 

The analysis of those constructions is quite intricate, and the dimensions required are larger, even more so when dropping the assumption that the iterates stay in the span of the queried derivatives (as assumed, e.g. by \citet{gu:lower}). We will first provide a sketch of the main argument, and then go on to state the key results.

To reason about the construction, a very useful notion is that of a higher-order (robust) zero-chain \cite{carmon:lower:i}. 
\begin{definition}[Robust zero-chain]
	\label{def:robustchain}
	A function $f: \R^d \rightarrow \R$ is a robust zero-chain if for all $i\in [d]$ and $\x \in \R^d$ the following implication holds: If $\abs{x_j} < \frac{1}{2}$ for all $j \geq i$, then
\begin{align*}
\forall \y \in N(\x) \,:\,f(\y) = f(y_1,\ldots, y_i, 0,\ldots,0),
\end{align*}
	where $N(\x)$ denotes an open neighborhood of $\x$.
\end{definition}

One can observe that the partial derivatives of such a function $f$ at $\x$ are zero for all indices $j > i$, which is the key to ensure that the oracle responses give away information slowly. 

Recall the function $\bar{f}_K = f_{K,\mathbf{1}}$ from Definition~\ref{def:deterministic:hard_familly}. This function is a robust zero-chain for any $K \geq 1$ \cite{carmon:lower:i}. Since it also has the desirable property that its gradient is large as long as there remain coordinates which are close to zero, we can exploit copies of it in a lower bound construction. Instead of using a single matrix $\V$ to rotate the input adversarially, we will follow \citet{fang:spider} and use $n$ different matrices $\B_i$ with orthogonal columns drawn at random and prove that with some fixed probability for a large number of iterates
$$
    \inner{\bb_{i,K}}{\x^{(t)}} < 1/2.
$$
This will (similarly as in the deterministic case) imply that the gradient of function $i$ is bounded from below. We will refer to the process of the algorithm finding inputs that make these inner products large as ``discovering'' coordinates.

\begin{definition}[Finite-sum hard distribution]
	\label{def:finitesumhard}
	Let $n,K \in \N$. Let $d \in \N$ be divisible by $n$ and let $R = 230\sqrt{K}$. 
	Then draw $\B = \big[\B_1 \,\vert\, \cdots \,\lvert\, \B_n \big] \in \ortho(d/n, nK)$
	uniformly at random and let $
	\C = \big[\C_1 \,\vert \,\cdots \,\lvert\, \C_n \big] \in \ortho(d, d)
	$ be arbitrary. We define our unscaled hard instance with finite-sum structure as
	\begin{align*}
	F^*\, : \,\,\, &\R^{d} \rightarrow \R \\
	&\x \mapsto \frac{1}{n}\sum_{i=1}^{n} f^*_i(\x) = \frac{1}{n}\sum_{i=1}^{n}\hat{f}_{K;\B_i}(\C_i^T\x),
	\end{align*} 
	where we define
	$$
	    \hat{f}_{K;\B_i}(\y) := \bar{f}_K\left(\B_i^T\rho\left(\y\right)\right) + \frac{1}{10}\norm{\y}^2
	$$
	and 
	$$
		\rho(\y) := \frac{\y}{\sqrt{1 + \norm{\y}^2/R^2}}.
	$$
	Because of the random choice of matrix $\B$, this induces a distribution.
\end{definition}

Note that the same construction has been used in \citet{fang:spider} to show a lower bound for the first-order mean-squared setting and that the last two definitions are originally due to \citet{carmon:lower:i}. A brief discussion is in order: the purpose of $\B_i$ is as discussed above, namely using the zero-chain property of $\bar{f}_K$ to make sure that any algorithm has a hard time discovering coordinates. The composition with $\rho$ ensures that an algorithm cannot simply make the iterates large to learn coordinates, and $\C_i$ will be useful to bound the gradient norm of $F$ in terms of the $\norm{\nabla f_i}$'s: exactly what we need for a lower bound.

Our goal is to derive lower bounds for any possible Lipschitz constants and optimality gaps. This means that we will scale $F^*$ to meet the various requirements. The notion of a function-informed process \cite{carmon:lower:i} will permit us to reason about a scaled version of our function $F^*$ while thinking about what another algorithm would do on the unscaled $F^*$.

\begin{definition}[Function-informed process]
	We call a sequence of indices and iterates $\{[i^t, x^{(t)}]\}_{t \in \N}$ informed by a function $F$ if it follows the same distribution as $\algo[F]$ for some randomized $\algo$.
\end{definition}
\begin{lemma}
	\label{obs:informedprocess}
	Let $F$ be an instance of a finite-sum optimization problem. Let $a,b >0$. Consider the function $G(\x) = aF(\x/b)$ and assume $\{[i^t, \x^{(t)}]\}_{t \in \N}$ is produced by $\algo$ on function $G$, i.e. $\algo[G] = \{[i^t, \x^{(t)}]\}_{t \in \N}$. Then $\{[i^t, \x^{(t)} / b]\}_{t \in \N}$ is informed by $F$.
\end{lemma}

Above, we hinted at the importance of small inner products of the iterates with the columns of the matrices $\B_i$. This intuition is formalized in the lemma below, that also appears in \citet{fang:spider} (for first-order algorithms) and closely resembles key lemmas in prior work \cite{woodworth:tight, carmon:lower:i}.
\begin{lemma}
	\label{lemma:fangkey}
	Let $\{[i^t,\x^{(t)}]\}_{t\in \N}$ be informed by $F^*$ drawn from the distribution in Definition \ref{def:finitesumhard}, let $\delta \in (0,1)$ and $T = \frac{nK}{2}$. 
	For any $t \in [T]$ and $i \in [n]$, let $I_i^{t-1}$ be the number of occurrences of index $i$ in $i^{0:t-1}$, i.e. the number of queries with index $i$ up to iteration $t$ (the iteration producing $\x^{(t)}$). Let $I_i^{-1} = 0$ by default. For any $t \in \{0,\ldots,T\}$ define  $\mathcal{U}_i^{(t)}$ to be the set of the last $K - I_i^{(t-1)} $ columns of $\B_i$ (provided $K - I_i^{t-1} \geq 1$, otherwise the set is empty). More formally
	$$
		\mathcal{U}_i^{(t)} := \left\{ \bb_{i,I_i^{t-1}+1} , \ldots , \bb_{i,K} 
			\right\}.
	$$
	
	Then the following holds for some constant $c_0 < \infty$:
	if $d \geq c_0 n^3K^2 \log(\frac{n^2K^2}{\delta})$, then with probability at least $1-\delta$ we have $\forall   t \in \{0,\ldots, T\} ,\, \forall i \in [n] ,\, \forall \mathbf{u} \in \mathcal{U}_i^{(t)}$
	\begin{equation}
	\label{eq:small_prod}
	 \abs{\langle \mathbf{u}, \rho(\C_i^T \x^{(t)})\rangle} < \frac{1}{2}.
	\end{equation}
\end{lemma}
  To clarify the indexing, let us consider a concrete example. Fix the $10$th iteration and $j\in [n]$. Recall that to produce iterate $\x^{(10)}$ the algorithm has access to the derivative information for the first 10 iterates (up to iterate 9). If $j$ occurs 2 times in $i^{0:9}$, then $\mathcal{U}_j^{(10)} = \{\bb_j^{(3)},\ldots \bb_{j}^{(K)}\}$ and for all those elements the dot product bound (\ref{eq:small_prod}) holds.  Note that what we refer to as ``discovered" columns at iteration $t$ corresponds to the columns of $\B_i$ that are not in $\mathcal{U}_i^{(t)}$.

The key takeaway from Lemma \ref{lemma:fangkey} is that for each index $i\in [n]$ the algorithm needs $K$ queries to that index to learn all columns of $\B_i$. Consequently, the input of the zero-chain $\bar{f}_K$ stays small in absolute terms for the coordinates corresponding to columns in $\mathcal{U}_i^{(t)}$ with high probability. This is good because $\bar{f}_K$'s large-gradient property (Lemma \ref{lemma:original:largegradient}) then makes the gradient of $F^*$ large as well:

 \begin{lemma}
 	\label{lemma:largegradientrand}
 	Let $\{[i^t,\x^{(t)}]\}_{t\in \N}$ be informed by $F^*$ drawn from the distribution in Definition \ref{def:finitesumhard} and let $\delta \in (0,1)$.
 	Then the following holds for some numerical constant $c_0 < \infty$: 
 	if $d \geq c_0 n^3K^2 \log(\frac{n^2K^2}{\delta})$, with probability at least $1-\delta$ we have for all $t \in \{0,\ldots,T\}$
 	$$
 		\norm{\nabla F^*(\x^{(t)})} > \frac{1}{4\sqrt{n}}.
 	$$
 \end{lemma}

\subsection{Lower bound for the individual smooth setting}
To derive results for any incarnation of the function classes in Assumption~\ref{assumption:individual}, one can rescale the function and the inputs and use the above lemmas, exploiting the fact that they hold for function-informed processes. The analysis yields:

\begin{theorem}
	\label{theorem:individual_smoothness}
	For any randomized algorithm $\algo$ satisfying Assumption~\ref{assumption:algo}, $p \in \N$, $\Delta$, $L_p$, $\epsilon$ and $
	n \leq c_p \Delta^{\frac{2p}{p+1}} L_p^{\frac{2}{p+1}}\epsilon^{-2}
	$, there exists a dimension $d \leq \tilde{\upp}(n^{\frac{2p-1}{p}}\Delta L_p^{2/p} \varepsilon^{-\frac{2(p+1)}{p}}) \leq \tilde{\upp}(n^2\Delta L_p^2 \varepsilon^{-4})$ and a function $F \in {\mathcal{F}}_p^n(\Delta, L_p)$ such that
	$$
	T_\epsilon(\algo, F) \geq \Omega\left( \left(\frac{L_p}{\hat{\ell}_p}\right)^{\frac{1}{p}} \frac{\Delta {n}^{\frac{p-1}{2p}}}{\epsilon^{\frac{p+1}{p}}} \right),
	$$
	where $\hat{\ell}_p \leq \exp(cp \log p + c)$ for some constant $c < \infty$. For fixed $p$, $c_p$ is also a universal constant.
\end{theorem}

Our result is essentially a lower bound of $\Omega\left (  {n}^{\frac{p-1}{2p}}\epsilon^{-\frac{p+1}{p}}\right )$ for fixed $p$, up to constant factors. The increasing dependence on $n$ is consistent with the empirical observation that higher-order methods typically need to employ larger batch sizes (see Section 8.1.3 in \citet{goodfellow:deeplearning_book}), but it could also be an artefact of a not yet perfect analysis.

For second-order algorithms, the best rate with our individual smoothness assumption is achieved by \citet{zhou:svrc}. Their algorithm finds an approximate local minimum in $\tilde{\upp}(n^{4/5}\varepsilon^{-3/2})$ oracle calls. Our lower bound reads as $\low(n^{1/4}\varepsilon^{-3/2})$ for Assumption~\ref{assumption:individual} with $p=2$, which implies that our bound exhibits a rather large $\tilde{\upp}(n^{11/20})$ gap.
\subsection{A new assumption for second-order smoothness}
We point out that a similar $n^{1/2}$ gap is present in the case of $p=1$ \cite{gu:lower}, which remains an open problem. For the first-order setting, a way to get matching bounds is to use the first-order mean-squared smoothness assumption, yielding the optimal $\Theta(\sqrt{n}/\epsilon^{-2})$ oracle complexity \cite{fang:spider}. It has been observed by \citet{gu:lower} that this assumption is sufficient for a variety of first-order methods. This raises a natural question: is there a second-order analogue to mean-squared smoothness? The mean-squared assumption effectively controls the second moment of the random variable that arises when fixing $\x,\y$, drawing $f_i$ at random and considering $\nabla f_i(\x)- \nabla f_i(\y)$. For cubic regularization methods, a natural analogue is the \emph{third} moment of the Hessian difference.

In the following, we will show that one can indeed weaken the assumption of the SVRC algorithm from \citet{zhou:svrc} to Assumption~\ref{assumption:third}.

\begin{algorithm}[h]
	\caption{SVRC \cite{zhou:svrc}}
	\label{algorithm:svrc}
	\begin{algorithmic}
		\STATE {\bfseries Input:} Gradient and Hessian batch sizes $b_g$, $b_h$, cubic penalty parameter $M$, number of epochs $S$ and steps per epoch $T$. Starting point $\x_0$
		\STATE $\widehat{\x}^1 = \x_0$
		\FOR{$s=1$ {\bfseries to} $S$}
		\STATE $\x_0^s = \snap$
		\STATE $\snapgrad = \grdf(\snap)$, $\snaphess = \hesf(\snap)$
		\FOR{$t=0$ {\bfseries to} $T-1$}
		\STATE Sample index sets $I_g, I_h$, with $\abs{I_h} = b_h, \abs{I_g} = b_g$
		\STATE $\estim = \frac{1}{b_g}\sum_{i_t \in I_g} [\grdfi(\curir) - \grdfi(\snap)] + \snapgrad - (\frac{1}{b_g}\sum_{i_t \in I_g} \hesfi (\snap)- \snaphess)(\xdiff)$
		\STATE $\hestim = \frac{1}{b_h}\sum_{j_t \in I_h}[\hesfj(\curir) - \hesfj(\snap)] + \snaphess$
		\STATE $\step = \arg \min_{\mathbf{h}}[\inner{\estim}{\mathbf{h}} + \frac{1}{2}\inner{\hestim \mathbf{h}}{ \mathbf{h}} + \frac{M}{6}\norm{\mathbf{h}}^3]$
		\STATE $\x_{t+1}^s = \curir + \step$
		\ENDFOR
		\STATE $\widehat{\x}^{s+1} = \x^s_T$
		\ENDFOR
		\STATE {\bfseries Input:} $\x_{\mathrm{out}} = \x_t^s$, where $s \in [S]$, $t\in[T]$ are chosen uniformly at random. 
	\end{algorithmic}
\end{algorithm}

\begin{assumption}
	\label{assumption:third}
	We say a function $F = \sum_{i=1}^n f_i$ with $f_i: \R^d \rightarrow \R$ respects the third-moment smoothness assumption with constant $L_2$ if for any $\x, \y \in \R^d$
	$$
	\left( \E_i\norm{\nabla^2 f_i(\x) - \nabla^2 f_i(\y)}^3 \right)^{\frac{1}{3}} \leq L_2 \norm{\x-\y}.
	$$
	The expected value is taken w.r.t. a uniform distriubtion on $[n]$. We also assume $F$ satisfies Assumption~\ref{assumption:individual}~ii).
\end{assumption}
Note that this assumption is weaker than the usual second-order smoothness, but it is stronger than a second moment assumption, due to $\E[\abs{X}^s]^{1/s} \leq \E[\abs{X}^t]^{1/t}$ for $s < t$. Furthermore, through Jensen's inequality, it is easy to observe that $F$ has Lipschitz continuous Hessian, which is one reason why the assumption turns out to be useful. The second one is that error terms for cubic regularization are third powers, so this assumption provides a more natural fit than, say, a mean-squared Lipschitz assumption on the Hessian.

With some minor changes to the convergence analysis, the guarantees of SVRC (to second-order stationarity) can essentially be retained. A full proof is given in Appendix~\ref{appendix:C}.

\begin{theorem}
	\label{theorem:third_moment:upper}
	Let $M = C_M L_2$ for $C_M = 150$. Let the epoch length be $T = \max\{2,n^{1/5}\}$ and the number of epochs $S = \max\{ 1, 240 C_M^2 L_2^{1/2} \Delta n^{-1/5}\epsilon^{-3/2}\}$. Set the batch sizes to $b_g = 5 \max\{n^{4/5}, 2^4\}$ and $b_h = 3000 \max\{4, n^{2/5}\} \log^3 d$. Then SVRC under Assumption \ref{assumption:third} needs
	$$
	\tilde{\upp}\left(n+\frac{\Delta \sqrt{L_2} n^{4/5}}{\epsilon^{3/2}}\right)
	$$
	oracle queries to find a point $\x_{\mathrm{out}}$ such that, in expectation
	\begin{equation}
	\label{eq:mustuff}
	\max \left\{ \norm{\grdf(\x_{\mathrm{out}})}^{3/2}, -\frac{\lambda_\mathrm{min}^3(\hesf(\x_{\mathrm{out}}))}{L_{2}^{3/2}}\right\} \leq \epsilon^{3/2}.
	\end{equation}
	In particular it holds that
	$$
	\E \norm{\nabla F(\x_{\mathrm{out}})} \leq \varepsilon.
	$$
\end{theorem}
Note that if $\x$ satisfies \eqref{eq:mustuff}, then $\x$ is an approximate local minimum of $F$ \footnote{This is a point such that $\norm{\nabla F(\x)} \leq\epsilon$ and $\lambda_{\mathrm{min}}(\nabla^2F(\x)) \geq -\sqrt{L_2\epsilon}$}. If one compares this theorem to Theorem 6 and Corollary 9 in \citet{zhou:svrc}, one notices that the minimum batch size is larger by a polylogarithmic factor. This is indeed due to the fact that under the new smoothness assumption, bounding the maximum Hessian difference can only be done through bounding the sum, unlike before. It seems possible that this dependency can be removed by using more suitable moment inequalities for matrices than the ones proposed in the original proof. 

What is now left to do is to provide a tighter lower bound. Indeed, the following holds:
\begin{theorem}
	\label{theorem:third_moment:lower}
	For any randomized algorithm $\algo$ satisfying Assumption~\ref{assumption:algo}, $\Delta$, $L_2$, $\epsilon$, and $n \leq \frac{c\Delta^{12/7} L_2^{6/7}}{\epsilon^{18/7}}$ there exists a dimension $d \leq \tilde{\upp}(n^2\Delta L_2 \varepsilon^{-3})$ and a function $F = \frac{1}{n}\sum_{i=1}^{n}f_i$ that satisfies Assumption \ref{assumption:third} such that
	$$
	T_\epsilon(\algo, F) \geq \Omega\left( \frac{L_2^{1/2}\Delta n^{5/12}}{\epsilon^{\frac{3}{2}}} \right),
	$$
	where the constants hidden by $\Omega$ do not depend on $\varepsilon$ or $n$. $c$ is also a universal constant.
\end{theorem}
Note the $n^{1/6}$ difference when compared to Theorem~\ref{theorem:individual_smoothness}. The reason for this is that the tall orthogonal matrices $\C_i$ used in the construction allow a function satisfying Assumption~\ref{assumption:individual} to be scaled by $\sqrt[3]{n}$ and still respect Assumption~\ref{assumption:third}. With this, the $\sqrt{L_2}$ dependence of the lower bounds in Theorem~\ref{theorem:individual_smoothness} explains this $n^{1/6}$ difference.

So -- to conclude -- under Assumption~\ref{assumption:third} and $p=2$, one can find an $\epsilon$-approximate local minimum in $\tilde{\upp}(n^{4/5}\epsilon^{-3/2})$ oracle queries while the lower bound lies at $\Omega(n^{5/12}\epsilon^{-3/2})$. While the gap remains at $\Omega(n^{23/60})$, this is a notable improvement over the results for Assumption~\ref{assumption:individual}, which means that the third-moment smoothness assumption gets us closer to understanding the fundamental limits for higher-order variance-reduced methods.
\section{Discussion}
In this work, we have analyzed the oracle complexity of higher-order smooth non-convex finite-sum optimization. We have shown that speedup (e.g. through variance reduction) in the non-convex case, as in the convex case, requires randomization. 

For randomized algorithms, the picture remains unclear: while we are able to show non-trivial lower bounds -- i.e. $n$ does not vanish, unlike in the $p=1$ case -- our bounds are not tight. The gaps that remain to be closed are of similar approximate $\tilde{\upp}(\sqrt{n})$ magnitude for first and second-order algorithms and considering a moment-based smoothness assumption yields tighter bounds in both cases. It remains unclear whether these smoothness assumptions are equivalent for algorithmic purposes, or if individual smoothness is stronger than mean-squared/third-moment smoothness.

Algorithmic results for gradient based algorithms seem to either indicate a failure to exploit that \emph{every} component is smooth or hint at the fact that the lower bound results from Theorem~\ref{theorem:individual_smoothness} and \citet{gu:lower} could be improved for all orders of smoothness.

 There are a few directions of improvement for the specific problem of second-order algorithms we would like to mention. Firstly, there may be different models that allow to better characterize optimal oracle complexities. Indeed, some of the most recent algorithms from Section~\ref{sec:higherorder_algo} prioritize Hessian complexity, and achieve a complexity of $\tilde{\upp}(\sqrt{n}/\epsilon^{3/2})$ at the cost of more gradient queries. It would be interesting to derive lower bounds for a setting where gradient and Hessian complexities are counted separately, perhaps traded off in a flexible way. Secondly, it is plausible that a stronger lower bound can be achieved by analyzing SOSPs instead of FOSPs, as it is the typical guarantee. However, we do not believe that this is the key challenge, because the underlying issues to obtain stronger bounds seem to be the same for both first- and second-order methods.

In any case, further research is needed to fully understand the achievable oracle complexities of variance reduced methods.

\paragraph{Acknowledgments}
We are grateful to anonymous reviewers for their helpful comments.

\bibliography{refs}
\bibliographystyle{apalike}

\newpage
\setcounter{section}{0}
\renewcommand{\thesection}{\Alph{section}}


\newcommand{\assfunc}{2.1}
\newcommand{\assoracle}{2.2}
\newcommand{\assalgo}{2.3}

\newcommand{\defdeterfamily}{3.1}
\newcommand{\defdeterinstance}{3.2}
\newcommand{\lemdeterconsistency}{3.3}
\newcommand{\lemoriginallargegrad}{3.4}
\newcommand{\lemdeterlargegrad}{3.5}
\newcommand{\theodeter}{3.6}

\newcommand{\defrandrobust}{4.1}
\newcommand{\defrandhard}{4.2}
\newcommand{\defrandinformed}{4.3}
\newcommand{\lemmainformed}{4.4}
\newcommand{\lemmakey}{4.5}
\newcommand{\lemmagradbound}{4.6}
\newcommand{\theorand}{4.7}
\newcommand{\assthird}{4.8}
\newcommand{\theothirdup}{4.9}
\newcommand{\theothirdlow}{4.10}

\section{Lower bounds for deterministic algorithms}
The appendix is structured in 4 parts. Appendix A provides all omitted proofs for Section~3, while Appendix B provides the same for Section~4, up to the end of 4.1. In Appendix C we give the proofs for Theorems \theothirdup{} and \theothirdlow{}. Finally Appendix D contains the proof of a simple observation that is needed for all constructions.

\newcommand{\fdet}{f_{K,\ddelta}}
\renewcommand{\v}{\mathbf{v}}
\newcommand{\seqgam}[1]{#1_{j+\gamma_1}\cdots #1_{j+\gamma_p} #1_j}

\renewcommand\ddelta{\bm{\delta}}
\newcommand\ggamma{\bm{\gamma}}

\subsection{Proof of Theorem \theodeter}
Along with Lemma~\lemdeterlargegrad{}, we need the following result that will allow us to ensure $F$ satisfies Assumption~\assfunc{}.
\begin{lemma}
\label{lemma:deterministic:properties}
For all $K$ and $\ddelta \in \{0,1\}^K$, the function $f_{K,\ddelta}$ from Definition~\defdeterfamily{} satisfies
\begin{enumerate}[i)]
\item The initial sub-optimality can be bounded by $f_{K,\ddelta}(\mathbf{0}) - \inf_{\x \in \R^K}f_{K,\ddelta}(\x) \leq 12K$.
\item The function is $p$-th order $\ell_p$-smooth with $\ell_p \leq \exp(\frac{5}{2}p \log p + cp)$ for some numerical constant $c < \infty$.
\end{enumerate}
\end{lemma}

\begin{proof}{of Theorem~\theodeter}
At this point, we are ready to proceed with our argument. Recall Definition \defdeterinstance{} of the hard instance $F = \frac{1}{n}\sum f_i$ with $\V \in \ortho(d, K+1)$:
$$
   f_i(\x) = \lambda \sigma^{p+1} f_{K+1,\boldsymbol\delta_i}(\V^T\x / \sigma).
$$
We will guarantee smoothness through $\lambda$, bound the gradient norm from below through $\sigma$ and finally control the distance to optimality with $K$.
By Lemma~\ref{lemma:additional:tensorineq} and Lemma~\ref{lemma:deterministic:properties}ii), we can write for any $\x, \y \in \R^d$
\begin{align*}
 \norm{\nabla^p  f_i(\x) - \nabla^p  f_i(\y)} &\leq \lambda \ell_p \norm{\V^T(\x-\y)} \leq \lambda \ell_p \norm{\x-\y},
\end{align*}
where the second inequality follows because $d\geq K+1$ and because we can complete $\V$ to be a square orthogonal matrix. We see that the choice $\lambda = L_p/\ell_p$ guarantees $p$th-order smoothness with constant $L_p$.

Next, we will turn to bounding the gradient from below.
By Lemma~\lemdeterlargegrad{}, we can lower bound $\norm{\nabla F(\x^{(t)})} > \frac{\lambda \sigma^p}{4}$ for all iterates up to the end of round $K+1$. We desire a lower bound for $\epsilon$-stationarity, so we will choose $\sigma = \left (\frac{4\varepsilon \ell_p}{L_p} \right)^{\frac{1}{p}}$. 

As a last step, we will choose $K$ such that the initial gap on suboptimality is bounded by $\Delta$. For that, we use Lemma~\ref{lemma:deterministic:properties}i).
We want
\begin{align*}
    F(0) - \inf_{\x \in \R^d} F(\x) &\leq \frac{1}{n}\sum_{i=1}^n \left[ f_i(\mathbf{0}) - \inf_{\x \in \R^d} f_i(\x) \right] \\&\leq \frac{\lambda \sigma^{p+1}}{n}\sum_{i=1}^n \left[ f_{K+1,\boldsymbol\delta_i}(\mathbf{0}) - \inf_{\y \in \R^{K+1}}f_{K+1,\boldsymbol\delta_i}(\y)  \right] \\&\leq 12 \lambda \sigma^{p+1} (K+1) \leq \Delta.
\end{align*}
As a larger value of $K$ yields a better bound, we can choose 
$$
    K+1 = \left \lfloor \frac{\Delta}{192} \left( \frac{L_p}{\ell_p}\right)^{\frac{1}{p}} \frac{1}{\epsilon^{\frac{p+1}{p}}} \right \rfloor \leq \frac{\Delta}{12} \frac{\ell_p}{L_p}  \left (\frac{L_p}{4\varepsilon \ell_p} \right)^{\frac{p+1}{p}}.
$$
Because $K$ is the number of rounds and each round consists of $\Omega(n)$ queries, this yields a $\Omega \left( \left(\frac{L_p}{\ell_p}\right)^{1/p}\frac{\Delta n}{\varepsilon^{(p+1)/p}}\right)$ lower bound, as desired. As explained in the proof of Lemma~\lemdeterconsistency{}, the dimension $d$ must merely be larger than the sum of the lower bound and the number of rounds, i.e. linear in the lower bound. This completes the proof.
\end{proof}

\subsection{Proof of technical lemmas for the deterministic setting}
\begin{proof}{of Lemma \lemdeterlargegrad{}}
By Lemma~\lemdeterconsistency{}, we have $(\V^T \x^{(t)})_{K+1} = 0$ for all iterates up until the end of round $K+1$ and will therefore be able to apply Lemma~\lemoriginallargegrad{}. We use $\tilde{\nabla}$ to denote the gradient with respect to $\V^T\x/\sigma$ and write
\begin{align*}
 \lambda \sigma^{p+1} \tilde{\nabla} \left[ \frac{1}{n} \sum_{i=1}^n f_{K+1,\boldsymbol\delta_i}(\V^T\x/\sigma) \right] = \frac{1}{n}\left \lceil \frac{n}{2} \right \rceil \lambda \sigma^{p+1} \tilde{\nabla} \left[ \bar{f}_{K+1}(\V^T\x/\sigma) \right]. \\
\end{align*}
Using the chain rule, we see that
$$
\nabla F(\x) = \frac{1}{n}\left \lceil \frac{n}{2} \right \rceil \lambda \sigma^{p} \V \tilde{\nabla} \left[ \bar{f}_{K+1}(\V^T\x/\sigma) \right], \\
$$ and thus by Lemma~\lemoriginallargegrad{} and by the fact that $\V^T \V = \mathbf{I}_{K+1}$
$$\norm{\nabla F(\x)} \geq \frac{\lambda \sigma^p}{4}.$$
\end{proof}

To show Lemma~\ref{lemma:deterministic:properties}, we need the following technical result, which is a subset of Lemma~1 in \citet{carmon:lower:i}.

\begin{lemma}
\label{carmon:technical_lemma}
For the functions from Definition~\defdeterfamily{} we have
\begin{enumerate}[i)]
\item Both $\Psi$ and $\Phi$ are infinitely differentiable, and for all $q \in \N$ we have
$$
    \sup_x \, \lvert \Psi^{(q)}(x) \rvert  \leq \exp \left( \frac{5q}{2} \log (4q) \right)  
    \quad  \text{and} \quad 
    \sup_x \, \lvert \Phi^{(q)}(x) \rvert  \leq \exp \left( \frac{3q}{2} \log \frac{3q}{2} \right).
$$
\item The functions and derivatives $\Psi$, $\Psi'$, $\Phi$, $\Phi'$ are non-negative and bounded, with
$$
    0 \leq \Psi < e, \quad 0 \leq \Psi' \leq \sqrt{54/e}, \quad 0 < \Phi < \sqrt{2\pi e} \quad \text{and} \quad 0 < \Phi' \leq \sqrt{e}.
$$
\end{enumerate}
\end{lemma}

Now we present our proof, closely following \citet{carmon:lower:i}, Appendix B.2. We account for the indicators $\delta_k$ used in our construction, validating that they do not affect the aforementioned properties. \\

\begin{proof}{of Lemma~\ref{lemma:deterministic:properties}}
Fix $K \in \N, \ddelta \in \{0,1\}^K$. We first bound the suboptimality gap. We have $f_{K,\ddelta}(\mathbf{0}) \leq 0$ because $-\delta_1 \Psi(1)\Phi(0) \leq 0$ by Lemma \ref{carmon:technical_lemma} ii). By the same arguments, for any $\x$, we have 
\begin{align*}
    f_{K,\ddelta}(\x) 
    & = -\delta_1 \Psi(1)\Phi(x_1) + \sum_{k=2}^K \delta_k[\Psi(-x_{k-1})\Phi(-x_k) -\Psi(x_{k-1})\Phi(x_k)] \\
     &\geq  -\delta_1 \Psi(1)\Phi(x_1) - \sum_{k=2}^K \delta_k[\Psi(x_{k-1})\Phi(x_k)]
     \geq  -\delta_1(e \cdot \sqrt{2\pi e}) - \sum_{k=2}^K \delta_k[e \sqrt{2\pi e}]
     > -K \cdot e \cdot \sqrt{2\pi e} \geq -12 K.
\end{align*}
Thus, we get our bound on suboptimality. \\
For the second part, let $\x \in \R ^K$. For a unit vector $\v \in \R ^K$ we define the directional projection $h_{\x,\v}(\theta) = \fdet(\x + \theta \v)$. It suffices to show that $\lvert h_{\x,\v}^{(p+1)}(0)\rvert \leq \ell_p$ for any $\x,\v$, because the directional projection is infinitely differentiable, by Lemma \ref{carmon:technical_lemma}. Fix $\x,\v \in \R^K$. We can write
$$
     h_{\x,\v}^{(p+1)}(0) = \sum_{j_1,\ldots,j_{p+1}=1}^K \partial_{j_1}\cdots \partial_{j_{p+1}}f_{K,\ddelta}(\x)v_{j_1}\cdots v_{j_{p+1}}.
$$
All multiplicative terms in $\fdet$ have zero derivatives unless all derivatives are w.r.t.\ adjacent indices. Defining for convenience $v_0 = v_{K+1} = 0$ we can express the above as
$$
    h_{\x,\v}^{(p+1)}(0) = \sum_{\ggamma \in \{0,1\}^p \cup \{0,-1\}^p} \sum_{j=1}^K \seqgam{\partial} \fdet(\x) \seqgam{v}.
$$
We can bound
$$
\max_{j\in[K]}\max_{\gamma \in \{0,1\}^p \cup \{0,-1\}^p} \lvert \seqgam{\partial}\fdet(\x)\lvert \leq 
\max_{k\in [0:k+1]} \left\{4 \sup_{y \in \R}\left \lvert \Psi^{(k)}(y) \right \rvert \sup_{y' \in \R} \left \lvert \Phi^{(p+1-k)}(y') \right \rvert \right\}.
$$ Here, we have used that $\delta$ can only (potentially) suppress terms and that there are only 4 terms which may involve partial derivatives with respect to either $x_j$ and $x_{j+1}$ or $x_j$ and $x_{j-1}$. Note that if $\ggamma \not= \mathbf{0}$, there are only 2 such terms. \\
With Lemma \ref{carmon:technical_lemma}, the above can be further bounded by
$$
4\sqrt{2\pi e}\cdot e^{2.5(p+1)\log(4(p+1))} \leq e^{2.5p+\log p + 4p +10}.
$$
We define $\ell_p = 2^{p+1}e^{2.5p+\log p + 4p +10} \leq e^{2.5p+\log p + 5p + 11}$. Finally, we can bound the quantity of interest
$$
    \lvert h_{\x,\v}^{(p+1)}(0) \rvert \leq \sum_{\ggamma \in \{0,1\}^p \cup \{0,-1\}^p} 2^{-(p+1)}\ell_p \left\lvert \sum_{j=1}^K  \seqgam{v} \right\rvert \leq \ell_p,
$$ because $\left\lvert \sum_{j=1}^K  \seqgam{v} \right\rvert \leq 1$, which follows from $\v$ being a unit vector (see \citet{carmon:lower:i}, B.2). This concludes the proof.
\end{proof}

\section{Lower bounds for randomized algorithms}
\subsection{Proof of Theorem~\theorand{}}
The function $\hat{f}_{K;\B_i}$ from Definition~\defrandhard{} has some very useful properties regarding its Lipschitz constants and its gap to optimality.
\begin{lemma}[Lemma 6 in \citet{carmon:lower:i}]
	\label{lemma:carmon:hat:properties}
	The function $\hat{f}_{K;\B_i}$ satisfies the following properties:
	\begin{enumerate}[i)]
		\item $\hat{f}_{K;\B_i}(\mathbf{0}) - \inf_{\y \in \R^{d/n}} \hat{f}_{K;\B_i}(\y)\leq 12 K$.
		\item For every $p \geq 1$, the $p$th-order derivatives of $\hat{f}_{K;\B_i}$ are $\hat{\ell}_p$-Lipschitz continuous, where $\hat{\ell}_p \leq \exp(cp \log p + c)$ for a numerical constant $c < \infty$.
	\end{enumerate}
\end{lemma}

With this, we can proceed with the proof of the main lower bound theorem.

\begin{proof}{of Theorem \theorand{}}
	We define the functions
	$$
	f_i(\x) = {\lambda \sigma^{p+1}} f^*_i\left(\frac{\x}{\sigma} \right) =  \lambda \sigma^{p+1} \hat{f}_{K;\B_i} \left(\frac{\C_i^T\x}{\sigma} \right),
	$$
	giving us
	$$
	F(\x) = \frac{1}{n}\sum_{i=1}^n f_i(\x).
	$$
	We will choose the scaling parameters to ensure that our instance belongs to the desired function class. We have
	\begin{align}
	\norm{\nabla^pf_i(\x) - \nabla^pf_i(\y)} \nonumber
	&\leq   \lambda \hat{\ell}_p \norm{\C_i^T \x - \C_i^T \y} \nonumber \\
	\label{eq:wasteful} &\leq  \lambda\hat{\ell}_p \norm{\x -  \y}  .
	\end{align}
	The first inequality follows from Lemmas \ref{lemma:additional:tensorineq} and \ref{lemma:carmon:hat:properties} and the second holds because $\C_i$ can be extended to the orthonormal matrix $\C$. The choice $\lambda = \frac{L_p}{\hat{\ell}_p}$ accomplishes our goal of smoothness with parameter $L_p$.
	
	Now fix an algorithm $\algo$ and assume $\{[i^t,\x^{(t)}]\}_{t\in \N}$ are the iterates produced by $\algo$ on $F$. Consequently, by Lemma~\lemmainformed{} $\{[i^t,\x^{(t)}/\sigma]\}_{t\in \N}$ is informed by $F^*$. 
	Therefore, we can apply Lemma~\lemmagradbound{} on the sequence $\{[i^t,\x^{(t)}/\sigma]\}_{t\in \N}$ to bound
	\begin{align*}
	\norm{\nabla F(\x^{(t)})}^2 
	&= \norm{ \lambda \sigma^{p} \nabla F^*(\x^{(t)}/\sigma)}^2  \\
	&= \lambda^2 \sigma^{2p} \norm{ \nabla F^*(\x^{(t)}/\sigma)}^2  \\
	&\stackrel{\lemmagradbound{}}{\geq} \frac{\sigma^{2p}\lambda^2}{16n},
	\end{align*}
	for all $t \in [0:T]$ with probability $1-\delta$ for a sufficiently large dimension $d$ (that depends on $\delta$). We will fix this dimension at the end. To get a lower bound for an $\varepsilon$ precision requirement we can choose $\sigma$ to be
	$$
	\frac{\sigma^{p}\lambda}{4\sqrt{n}} = \epsilon \iff \sigma = \left(\frac{4\sqrt{n}\varepsilon \hat{\ell}_p}{L_p}\right)^{\frac{1}{p}}.
	$$
	As a last step, we will guarantee the optimality gap requirement. 
	From Lemma~\ref{lemma:carmon:hat:properties}, we immediately have
	\begin{align*}
	F(\mathbf{0}) - \inf_{\x \in \R^d} F(\x)
	\leq 12\lambda \sigma^{p+1}K.
	\end{align*}
	We require 
	$$
	 12\lambda \sigma^{p+1}K = 12  \frac{L_p}{\hat{\ell}_p} \left(\frac{4\sqrt{n}\varepsilon \hat{\ell}_p}{L_p}\right)^{\frac{p+1}{p}} K\leq \Delta.
	$$
	To get the best possible bound, we choose
	$$
	K = \left \lfloor  \frac{\Delta}{192} \left(\frac{L_p}{\hat{\ell}_p}\right)^{\frac{1}{p}} \frac{1}{n^{\frac{p+1}{2p}}\epsilon^{\frac{p+1}{p}}}\right \rfloor.
	$$
	We will need that this $K$ is at least 1 in order to get a sensible bound, as becomes clear in the subsequent steps. To enforce this, we may require that
	$$
		\tilde{c}_p{\Delta}\left({L_p}\right)^{\frac{1}{p}} \frac{1}{\epsilon^{\frac{p+1}{p}}} \geq n^{\frac{p+1}{2p}},
	$$
	or in other words,
	$$
		n \leq c_p \frac{\Delta^{\frac{2p}{p+1}} L_p^{\frac{2}{p+1}}}{\epsilon^2}
	$$
	for some constants $c_p, \tilde{c}_p$ that depend on $p$. As Lemma~\lemmagradbound{} yields the lower bound $T = \frac{nK}{2}$ we get a lower bound of
	$$
	\Omega\left( \left(\frac{L_p}{\hat{\ell}_p}\right)^{\frac{1}{p}} \frac{\Delta {n}^{\frac{p-1}{2p}}}{\epsilon^{\frac{p+1}{p}}} \right)
	$$
	with probability  at least $1/2$ for large enough dimension $d$ (see below). Thus there must be a fixed function $F$ such that for this many iterations -- with probability $1/2$ depending only on $\xi$ -- the iterates $\algo$ produces on $F$ all have gradient larger than $\varepsilon$. For the dimension requirement, one can plug in the values of $K$ and $\delta = 1/2$ into the dimension requirement of Lemma~\lemmagradbound{}, to see that some $d \in \tilde{\upp}(n^{\frac{2p-1}{p}}\Delta L_p^{2/p}\varepsilon^{-\frac{2(p+1)}{p}}) \leq \tilde{\upp}(n^2\Delta L_p^2\epsilon^{-4})$ suffices.
\end{proof}

\subsection{Proof of technical lemmas for the randomized setting}
\begin{proof}{of Lemma \lemmainformed{}}
	We have $\nabla^{p}G(\x) = \frac{a}{b^p}\nabla^{p}F(\x/b)$. We have to exhibit an algorithm $\mathsf{B}$ such that $\mathsf{B}[F]$ follows the same distribution as $\{[i^t, \x^{(t)} / b]\}_{t \in \N}$.
	
	Let $\{A^{(t)} \}_{t\in \N}$ be the sequence of mappings that produce the iterates of $\algo$. With some mild abuse of notation we may write \footnote{We use $\nabla^k g_{i^{0:t-1}}(\x^{(0:t-1)})$ to denote the sequence of all queried $k$th-order derivatives to produce iterate $t$.}
	\begin{align*}
	\algo_\xi[G]^{(t)}= \,\,& A^{(t)} \Big \{
	&& \xi, i^{0:t-1}, \x^{(0:t-1)},   \\
	& && \nabla^{(0:q)}g_{i^0}(\x^{(0)}), \ldots ,\nabla^{(0:q)}g_{i^{t-1}} \left (\x^{(t-1)}\right ) 
	&&&\Big\}
	\\ = 
	\,\,&A^{(t)} \Big \{
	&& \xi, i^{0:t-1},\x^{(0:t-1)},  g_{i^{0:t-1}}\left(\x^{(0:t-1)}\right ),  \\
	& && 
	\nabla g_{i^{0:t-1}}\left (\x^{(0:t-1)}\right ), \ldots,
	\nabla ^q g_{i^{0:t-1}}\left (\x^{(0:t-1)}\right ) 
	&&&\Big\}.
	\end{align*}
	$\mathsf{B}$ shall choose $i^0$ exactly like $\algo$ does. We define the sequence of mappings $\{B^{(t)} \}_{t\in \N}$ underlying $\mathsf{B}$ on arbitrary input $H =\frac{1}{n} \sum_{i=1}^{n} h_i$ as
	\begin{align*}
	& \mathsf{B}_\xi[H]^{(t)}\\  \\ = \,\, &B^{(t)}\Big \{ 
	&& \xi, i^{0:t-1},\y^{(0:t-1)}, h_{i^{0:t-1}}\left (\y^{(0:t-1)}\right ),  \\
	& &&  
	\nabla h_{i^{0:t-1}}\left (\y^{(0:t-1)}\right ), \ldots,
	\nabla ^q h_{i^{0:t-1}}\left (\y^{(0:t-1)}\right )  &&&\Big\} \\
	= \,\, &\frac{1}{b} A^{(t)}\Big \{
	&&\xi, i^{0:t-1},b\cdot \y^{(0:t-1)}, a \cdot h_{i^{0:t-1}}\left (\y^{(0:t-1)}\right ), \\
	& && 
	\frac{a}{b}\nabla h_{i^{0:t-1}}\left (\y^{(0:t-1)}\right ), \ldots, 
	\frac{a}{b^q}\nabla ^q h_{i^{0:t-1}}\left (\y^{(0:t-1)}\right )
	&&&\Big\},
	\end{align*}
	where we apply the outer division only on the iterates and not the indices. 
	We can check by induction that for a fixed random seed $\xi$, $\mathsf{B}_\xi[F]^{(t)} = \frac{\algo_\xi[G]^{(t)}}{b}$ for all $t \in \N$: The base case is clear as $i^0$ does not depend on any oracle queries and $\x^{(0)} = \mathbf{0}$ is deterministic. Now assume that the equality holds for all $t' < t$. Then
	\begin{align*}
	& \mathsf{B}_\xi[F]^{(t)}\\ 
	\stackrel{\mathrm{I.H.}}{=} \,\,&B^{(t)}\Big \{
	&& \xi, i^{0:t-1},\frac{\x^{(0:t-1)}}{b}, 
	f_{i^{0:t-1}}\left (\frac{\x^{(0:t-1)}}{b}\right ), \\
	& && \nabla f_{i^{0:t-1}}\left (\frac{\x^{(0:t-1)}}{b}\right ), \ldots,
	\nabla^q f_{i^{0:t-1}}\left (\frac{\x^{(0:t-1)}}{b}\right )  &&& \Big\} \\
	= \,\, &  \frac{1}{b} A^{(t)}\Big \{
	&& \xi, i^{0:t-1},b\cdot \frac{\x^{(0:t-1)}}{b}, 
	a \cdot f_{i^{0:t-1}} \left (\frac{\x^{(0:t-1)}}{b}\right ), \\
	& &&\frac{a}{b}\nabla f_{i^{0:t-1}}\left (\frac{\x^{(0:t-1)}}{b}\right ), \ldots,
	\frac{a}{b^q}\nabla ^qf_{i^{0:t-1}}\left (\frac{\x^{(0:t-1)}}{b}\right )
	&&& \Big\} \\ = 
	\,\,&\frac{1}{b}A^{(t)} \Big \{
	&&\xi, i^{0:t-1},\x^{(0:t-1)}, 
	g_{i^{0:t-1}}\left (\x^{(0:t-1)}\right ),\\
	& && \nabla g_{i^{0:t-1}}\left (\x^{(0:t-1)}\right ), \ldots,
	\nabla ^q g_{i^{0:t-1}}\left (\x^{(0:t-1)}\right ) 
	&&&  \Big\} \\
	= \,\, &  \frac{\algo_\xi[G]^{(t)}}{b}.
	\end{align*}
	Therefore $\mathsf{B}[F]$ follows the same distribution as $\{[i_t, \x^{(t)} / b]\}_{t \in \N}$ and so the sequence is informed by $F$, as desired.
\end{proof}

The proof of Lemma \lemmakey{} is mostly identical to Lemma 12 in \citet{fang:spider} \footnote{The main difference is that we formalize that (thanks to the robust zero-chain), it does not matter how many derivatives the algorithm has access to, hence the identical statement.} and similar to Lemma 4 in \citet{carmon:lower:i}. We give it in full here for completeness and to convince the reader that the result holds for higher-order algorithms as well. The reader accustomed to lower bounds for convex optimization will be familiar with the ideas involved (see Lemma 6 and 7 in \citet{woodworth:tight}).

\begin{proof}{of Lemma \lemmakey{}}
	First, we define quantities that we will use throughout the proof.
	
	Define $\y_i^{(t)} = \rho(\C_i^T\x^{(t)}) = \frac{\C_i^T\x}{\sqrt{1+\norm{\C_i^T\x}^2/R^2}}$. Then $\y_i^{(t)} \in \R^{d/n}$ satisfies $\norm{\y_i^{(t)}} \leq R$. Let $\mathcal{V}_i^{(t)}$ be the set of previous transformed iterates at index $i$ along with the discovered columns of $\B$ of after iteration $t$:
	$$
		\mathcal{V}_i^{(t)} = \{\y_i^{(0)},\ldots, \y_i^{(t)}\} \cup \bigcup_{j=1}^n\{\bb_{j,1},\ldots, \bb_{j,\min(K,I_j^{t})}\}.
	$$
	 Let $\mathcal{U}_i^{(t)}$ be defined as in the premise of Lemma~\lemmakey{} and denote by $\tilde{\mathcal{U}}_i^{(t)}$ its ``complement" (all other columns): 
	$$
		\tilde{\mathcal{U}}_i^{(t)} = \left\{ \bb_{i,1},\ldots,\bb_{i,\min(K,I_i^{t-1})} \right\}.
	$$
	Define $\mathcal{U}^{(t)} = \bigcup_{i=1}^n\mathcal{U}^{(t)}_i$ and $\tilde{\mathcal{U}}^{(t)} = \bigcup_{i=1}^n \tilde{\mathcal{U}}^{(t)}_i$
	and let $\P_i^{(t)}$ denote the orthogonal projection onto the span of $\mathcal{V}_i^{(t)}$.  Let $\P_i^{(t)\bot} = \I - \P_i^{(t)}$ be its orthogonal complement. Both of these are mappings from $\R^{d/n} \rightarrow \R^{d/n}$.
	
	Recall that our ultimate goal is to show that $\{[i^t,\x^{(t)}]\}_{t\in \N}$ being informed by ${F}^*$ implies that with probability $1-\delta$, for all $ t \in \{0,\ldots,T\}$, all $i \in [n]$ and all corresponding $\mathbf{u} \in \mathcal{U}_i^{(t)}$ the inequality
	\begin{equation}
		\label{eq:appendix:smallprod}
			 \abs{\langle \mathbf{u}, \y_i^{(t)}\rangle} < \frac{1}{2}
	\end{equation}
	
	holds. The case $t=0$ is obviously true, so from now on we focus on showing \eqref{eq:appendix:smallprod} for $t \geq 1$. We will first define an auxiliary event, show that it implies our result and then bound its probability. For any $ t \in [T]$ define the event
	$$
		G^t = \bigcup_{i \in [n]}\bigcup_{\u \in \mathcal{U}^{(t)}}\left\{ \abs*{\langle \u, \P_i^{(t-1)\bot}\y_i^{(t)} \rangle} < a \norm{\P_i^{(t-1)\bot}\y_i^{(t)}}  \right\},
	$$
	where $a = \min\left(\frac{1}{3(T+1)}, \,\frac{1}{2(1+\sqrt{3T})R}\right)$. Note that the union is over $\mathcal{U}^{(t)}$ and not $\mathcal{U}^{(t)}_i$. Let $G^{\leq t} = \cap_{j=1}^t G^{j}$. 
	We first show that $G^{\leq T}$ implies \eqref{eq:appendix:smallprod}.
	
	Assume $\mathcal{U}_i^{(t)} \not = \emptyset$, otherwise \eqref{eq:appendix:smallprod} holds trivially. For any $i \in [n]$, $t \in [T]$ and $\u \in \mathcal{U}^{(t)}_i$ we have
	\begin{align*}
		\abs{\langle \mathbf{u}, \y_i^{(t)}\rangle} &\leq \abs*{\langle  \mathbf{u} , \P_i^{(t-1)}\y_i^{(t)} + \P_i^{(t-1)\bot}\y_i^{(t)} \rangle} \\
		&<  \abs*{\langle  \mathbf{u} , \P_i^{(t-1)}\y_i^{(t)} \rangle}  + a \norm{\P_i^{(t-1)\bot}\y_i^{(t)}}  \\
		&\leq R \norm{\P_i^{(t-1)}\u}  + a R .
 	\end{align*}
 	In the second step we used $G^{\leq T}$ and in the third step we used Cauchy-Schwarz and the fact that $\P_i^{(t-1)}$ and $\P_i^{(t-1)\bot}$ are orthogonal projectors and therefore self-adjoint.
	If we manage to show $\norm{\P_i^{(t-1)}\u}  \leq \sqrt{3T}a =: b$ we are done, because the choice of $a$ then implies that $a R + R \norm{\P_i^{(t-1)}\u} \leq \frac{1}{2}$. 
	
	We will show this by induction over $t \in [T]$: Consider $t=1$ and let $i \in [n]$ be arbitrary. We have $\mathcal{V}_i^{(t-1)} = \mathcal{V}_i^{(0)} =  \{\y_i^{(0)}, \bb_{i^0,1} \} = \{\mathbf{0}, \bb_{i^0,1} \} $. Because $\u$ can be any column of $\B$ except $\bb_{i^0,1}$ we have $\P_i^{(t-1)}\u = 0$. For the induction step, another way to write the vectors in $\mathcal{V}_i^{(t-1)}$ is in the order they are discovered. That is, add to the set each iterate at $i$ and an additional column of $\B_{i^j}$ for the queried index $i^{j}$ at iteration $j$. We get the sequence
	$$
		\y_i^{(0)},\,\bb_{i^0,\min(I^0_{i^0},K)},\,\y_i^{(1)},\,\bb_{i^1,\min(I^1_{i^1},K)},\ldots,\,\y_i^{(t-1)},\,\bb_{i^{t-1},\,\min(I^{t-1}_{i^{t-1}},K)}.
	$$
	We will now apply the Gram-Schmidt procedure on these vectors. Remember that for a sequence of vectors $\vvv_i$ the Gram-Schmidt procedure (without normalization) constructs vectors
	\begin{align*}
		\u_1 &= \vvv_1 \\
		\u_2 &= \vvv_2 - \mathrm{proj}_{\u_1}(\vvv_2) \\
		& \,\,\,\vdots \\
		\u_k &= \vvv_k - \mathrm{proj}_{\u_1,\ldots,\u_{k-1}}(\vvv_k),
	\end{align*}
	where $\mathrm{proj}_S$ shall denote the projection on a set of vectors $S$.
	Applying this scheme to our sequence above, we get vectors \footnote{Where $\P_i^{(-1)} = \mathbf{0}_{d/n,d/n}$ is the zero matrix for convenience.}
	\begin{align*}
		\left\{ \y_i^{(z)} - \P_i^{(z-1)}\y_i^{(z)} \right\}_{z=0}^{t-1} = \left\{  \P_i^{(z-1)\bot}\y_i^{(z)} \right\}_{z=0}^{t-1}
	\end{align*}
	and 
	\begin{align*}
		&\left\{ \bb_{i^z,\min(I_{i^{z}}^{z}, K)} - \P_i^{(z-1)}\bb_{i^z,\min(I_{i^{z}}^{z}, K)}  - \mathrm{proj}_{\P_i^{(z-1)\bot} \y_{i}^{(z)}} \bb_{i^z,\min(I_{i^{z}}^{z}, K)}  \right\}_{z=0}^{t-1} \\
		=: &\left\{  \hat{\P}_i^{(z-1)\bot} \bb_{i^z,\min(I_{i^{z}}^{z}, K)} \right\}_{z=0}^{t-1} .
	\end{align*}
	We have $\mathrm{proj}_{\P_i^{(z-1)\bot}\y_{i}^{(z)}} = \frac{({\P}_i^{(z-1)\bot}\y_i^{(z)} )({\P}_i^{(z-1)\bot}\y_i^{(z)} )^T}{\norm{{\P}_i^{(z-1)\bot}\y_i^{(z)} }^2}$ and therefore write the projection $\hat{\P}_i^{(z-1)}$ onto $\mathcal{V}_i^{(z-1)} \cup \{\y_i^{(z)}\}$ as
	$$
		\hat{\P}_i^{(z-1)} = \P_i^{(z-1)} + \frac{({\P}_i^{(z-1)\bot}\y_i^{(z)} )({\P}_i^{(z-1)\bot}\y_i^{(z)})^T}{\norm{{\P}_i^{(z-1)\bot}\y_i^{(z)} }^2}.
	$$
	The orthogonalized vectors give us a basis in which we can write the norm $\norm{\P_i^{(t-1)}\u}^2$ as 
	\begin{align}
		\label{eq:inductive_sum_bound}
		 \sum_{z=0}^{t-1} \abs*{\left \langle \frac{{\P}_i^{(z-1)\bot}\y_i^{(z)}}{\norm{{\P}_i^{(z-1)\bot}\y_i^{(z)}}}, \u\right \rangle}^2
		+ \sum_{z=0, \, I_{i^{z}}^{z} \leq K}^{t-1} \abs*{\left \langle \frac{{\hat{\P}}_i^{(z-1)\bot}\bb_{i^z,I_{i^{z}}^{z}} }{\norm{\hat{\P}_i^{(z-1)\bot}\bb_{i^z,I_{i^{z}}^{z}} }}, \u\right \rangle}^2.
	\end{align}
	Note that the set we applied Gram-Schmidt on was not linearly independent so we may get $\mathbf{0}$-vectors. These do not influence the calculations, so we simply assume they are not present in \eqref{eq:inductive_sum_bound} from now on. The first term in \eqref{eq:inductive_sum_bound} is bounded by $t a^2$ by the induction hypothesis. Let $z$ be arbitrary but fixed and assume $I_{i^z}^z \leq K$. Recall the definition of $\mathcal{U}^{(t)}$. Then $\u = \bb_{i,j}$ for some $j > I_i^{t-1} \geq I_i^{z}$.  $\B$ has orthonormal columns and so $\u \, \bot \, \bb_{i^z,I_{i^{z}}^{z}}$. We will bound the second term in \eqref{eq:inductive_sum_bound} now:
	\begin{align}
		&\abs*{\left \langle {\hat{\P}}_i^{(z-1)\bot}\bb_{i^z,I_{i^{z}}^{z}}, \u\right \rangle} \nonumber \\
		=\,\,& 
		\abs*{\left \langle \bb_{i^z,I_{i^{z}}^{z}}- {\hat{\P}}_i^{(z-1)}\bb_{i^z,I_{i^{z}}^{z}}, \u\right \rangle} \nonumber \\
		=\,\,&\abs*{\left \langle \bb_{i^z,I_{i^{z}}^{z}} - {{\P}}_i^{(z-1)}\bb_{i^z,I_{i^{z}}^{z}} - \frac{({\P}_i^{(z-1)\bot}\y_i^{(z)} )({\P}_i^{(z-1)\bot}\y_i^{(z)})^T}{\norm{{\P}_i^{(z-1)\bot}\y_i^{(z)} }^2}\bb_{i^z,I_{i^{z}}^{z}}, \u\right \rangle} \nonumber \\
		\leq \,\, & 
		 \label{eq:app:second_term} \abs*{\left \langle {{\P}}_i^{(z-1)}\bb_{i^z,I_{i^{z}}^{z}} ,\u \right \rangle }+ \abs* {\left \langle \frac{{\P}_i^{(z-1)\bot}\y_i^{(z)} }{\norm{{\P}_i^{(z-1)\bot}\y_i^{(z)} }}, \u\right \rangle \left \langle \frac{{\P}_i^{(z-1)\bot}\y_i^{(z)}}{\norm{{\P}_i^{(z-1)\bot}\y_i^{(z)} }},\bb_{i^z,I_{i^{z}}^{z}}\right \rangle} ,
	\end{align}
	where in the last step we used $\u \, \bot \, \bb_{i^z,I_{i^{z}}^{z}}$ and the triangle inequality.
	For an orthonormal projector $\P$ and any vectors $\vvv,\u$ we have $\langle \P\vvv,\u \rangle = \langle \P\vvv,\P\u \rangle$. Therefore the left term in \eqref{eq:app:second_term} can be bounded by $b^2$ as follows:
	\begin{align}
	 \abs*{\left \langle {{\P}}_i^{(z-1)}\bb_{i^z,I_{i^{z}}^{z}} ,\u \right \rangle } &= 
	 \abs*{\left \langle {{\P}}_i^{(z-1)}\bb_{i^z,I_{i^{z}}^{z}} ,{\P}_i^{(z-1)}\u \right \rangle }  \nonumber \\
	 &\leq 
	 \norm{\P_i^{(z-1)}\bb_{i^z,I_{i^{z}}^{z}}}\norm{\P_i^{(z-1)}\u} \nonumber \\
	 &\leq b^2. \label{eq:bsquared_bound}
	\end{align}
	The last step holds because of the induction hypothesis. Indeed, we have $\u \in \mathcal{U}^{(t)} \subset \mathcal{U}^{(z)}$ and   $\bb_{i^z,I_{i^z}^{z}} = \bb_{i^z,I_{i^z}^{z-1}+1} \in \mathcal{U}_{i^z}^{(z)} \subset \mathcal{U}^{(z)}$.
	
	 Next, our assumption is that $G^{\leq T}$ happens and therefore $G^{z}$ as well. Using its definition twice on the right term in \eqref{eq:app:second_term} yields
	\begin{align}
	 &\abs* {\left \langle \frac{{\P}_i^{(z-1)\bot}\y_i^{(z)} }{\norm{{\P}_i^{(z-1)\bot}\y_i^{(z)} }}, \u\right \rangle \left \langle \frac{{\P}_i^{(z-1)\bot}\y_i^{(z)}}{\norm{{\P}_i^{(z-1)\bot}\y_i^{(z)} }},\bb_{i^z,I_{i^{z}}^{z}}\right \rangle} \leq a^2 \label{eq:app:asquaredbound}.
	\end{align}
	
	We bound the norm in the denominator of the right term in \eqref{eq:inductive_sum_bound} by
	\begin{align*}
		\norm{\hat{\P}_i^{(z-1)\bot}\bb_{i^z,I_{i^{z}}^{z}}}^2
		&=  \norm{\bb_{i^z,I_{i^{z}}^{z}}}^2-\norm{ \hat{\P}_i^{(z-1)}\bb_{i^z,I_{i^{z}}^{z}}}^2 \\
		&=\norm{\bb_{i^z,I_{i^{z}}^{z}}}^2-\norm{ {\P}_i^{(z-1)}\bb_{i^z,I_{i^{z}}^{z}}}^2 - \abs*{ \left \langle  \frac{\P_i^{(z-1)\bot}\y_i^{(z)}}{\norm{\P_i^{(z-1)\bot}\y_i^{(z)}}} , \bb_{i^z,I_{i^{z}}^{z}} \right \rangle}^2\\
		&\geq 1 - b^2 - a^2.
	\end{align*}
	 The first step is justified by the Pythagorean theorem because $\hat{\P}_i^{(z-1)\bot}\bb_{i^z,I_{i^{z}}^{z}}$ and $\hat{\P}_i^{(z-1)}\bb_{i^z,I_{i^{z}}^{z}}$ are orthogonal. The second follows by the Pythagorean theorem and the definition of $\hat{\mathbf{P}}^{(z-1)}$. For the inequality, we use the same arguments as in \eqref{eq:bsquared_bound} and \eqref{eq:app:asquaredbound}.
	
	We can return to \eqref{eq:inductive_sum_bound}. Recall that $b = \sqrt{3T}a$ and thus
	$a^2+b^2 = 3Ta^2 + a^2 = (3T + 1)a^2 \leq a$ by definition of $a$. We use this in step $(*)$ below:
	\begin{align*}
	\norm{\P_i^{(t-1)}\u}^2 &= \sum_{z=0}^{t-1} \abs*{\left \langle \frac{{\P}_i^{(z-1)\bot}\y_i^{(z)}}{\norm{{\P}_i^{(z-1)\bot}\y_i^{(z)}}}, \u\right \rangle}^2
	+ \sum_{z=0,\, I_{i^{z}}^{z} \leq K}^{t-1} \abs*{\left \langle \frac{{\hat{\P}}_i^{(z-1)\bot}\bb_{i^z,I_{i^{z}}^{z}} }{\norm{\hat{\P}_i^{(z-1)\bot}\bb_{i^z,I_{i^{z}}^{z}} }}, \u\right \rangle}^2 \\
	&\leq ta^2+t\frac{(a^2+b^2)^2}{1-(a^2+b^2)} \\
	&\stackrel{(*)}{\leq} ta^2 + t\frac{a^2}{1-a} \\
	&\leq 3Ta^2 \\
	&= b^2,
	\end{align*}
	where the last inequality holds because $a \leq 1/2$ and $t\leq T$. This concludes the induction. We have thus proven that $G^{\leq T}$ implies our result, namely that equation (2) holds for all $ t \in \{0,\ldots,T\}$, all $i \in [n]$ and all corresponding $\mathbf{u} \in \mathcal{U}_i^{(t)}$.
	
	We now derive an upper bound for the probability of the complement event $(G^{\leq T})^c$. Note that if $G^{\leq T}$ does not happen, then there is a smallest $t$ for which it fails. For convenience, let $G^{<1}$ be an event that always happens. Using a union bound, this argumentation is reflected by
	\begin{equation}
			\Pr((G^{\leq T})^c \leq \sum_{t=1}^{T} \Pr((G^{\leq t})^c  \, \vert \, G^{<t}). \label{eq:app:motherofunions}
	\end{equation}

  We will bound the probability $\Pr((G^{\leq t})^c \,  \vert \, G^{<t})$. For the remainder of the proof, we need matrices analogous to the sets $\mathcal{U}^{(t)}$ and $\tilde{\mathcal{U}}^{(t)}$. First define $\hat{i}^{t}$ to be the sequence $i^{0:t-1}$. Then let 
 $$
 \tilde{\U}_j^{\hat{i}^{t}} = \left[\bb_{j,1} \, \vert \, \cdots \, \vert \, \bb_{j,\min(K,I_j^{t-1})}\right],
 $$ 
 where $I_j^{t-1}$ is according to the sequence $\hat{i}^{t}$. Then define $\tilde{\U}^{\hat{i}^{t}} = [\tilde{\U}_1^{\hat{i}^{t}} \cdots \tilde{\U}_n^{\hat{i}^{t}} ]$. Similarly, we define the ``complement" matrices
	$$
	\U_j^{\hat{i}^{t}} = \left[\bb_{j,I_j^{t-1} + 1} \,\vert \, \cdots \,\vert\, \bb_{j,K}  \right].
	$$
	Note that for any $j$, one of $\U_j^{\hat{i}^{t}}$ or $\tilde{\U}_j^{\hat{i}^{t}}$ could potentially be empty. This will not be problematic in what follows.
	 Analogous to before ${\U}^{\hat{i}^{t}} = [{\U}_1^{\hat{i}^{t}} \cdots {\U}_n^{\hat{i}^{t}} ]$. Finally  $\bar{\U}^{\hat{i}^{t}} = [\tilde{\U}^{\hat{i}^{t}}, {\U}^{\hat{i}^{t}}]$ is a matrix with  all columns of $\B$, but in different order. For our event, by the law of total probability we have
	\begin{align}
		\label{eq:app:thingtobound}
		&\,\,\Pr((G^{\leq t})^c  \, \vert \, G^{<t}) \nonumber \\
		= &\sum_{\ih_0^t \in \hat{S}^t}\E_{\xi, \tilde{\U}^{\ih_0^t}}\left[ \Pr((G^{\leq t})^c  \, \vert \,  G^{<t},  \ih^t=\ih_0^t,\xi,\tilde{\U}^{\ih_0^t}) \, \Pr(\ih^t=\ih_0^t \, \vert \, G^{<t},\xi,\tilde{\U}^{\ih_0^t})\right].
	\end{align}
	In the rest, we show for all (fixed) $t$, $\xi_0,\tilde{\U}_0, \hat{i}_0^t$ a bound on the probability 
	\begin{align}
		\label{eq:appendix:prob:union}
		&\,\,\,\,\Pr((G^{\leq t})^c  \, \vert \, G^{<t}, \ih^t=\ih_0^t, \xi=\xi_0,\tilde{\U}^{\ih_0^t} = \tilde{\U}_0) \nonumber \\
		\leq & \sum_{\substack{i \in [n] \\\u \in \mathcal{U}^{(t)}}} \Pr \left( \abs*{\langle \u, \P_i^{(t-1)\bot}\y_i^{(t)} \rangle} \geq a \norm{\P_i^{(t-1)\bot}\y_i^{(t)}}  \, \vert \, G^{<t}, \ih^t=\ih_0^t,\xi=\xi_0,\tilde{\U}^{\ih_0^t} = \tilde{\U}_0 \right).
	\end{align}
	A bound on \eqref{eq:appendix:prob:union} is also a bound for \eqref{eq:app:thingtobound}, because 
	$$
		\sum_{\ih_0^t \in \hat{S}^t}\E_{\xi, \tilde{\U}^{\ih_0^t}} \Pr(\ih^t=\ih_0^t \, \vert \, G^{<t},\xi,\tilde{\U}^{\ih_0^t}) = 1.
	$$
	First, we show that given $G^{<t}$, the next iterate $[i^t, \x^{(t)}]$ produced by $\algo$ only depends on $\tilde{\U}^{\ih^t}$ and not the full draw of $\bar{\U}^{\ih^t}$, because $\bar{f}_K$ is a robust zero-chain. This is formalized below:
	\begin{lemma}
		\label{appendix:sublemma}
		For every $t \in [T]$, there exist measurable functions $A^{(t)}_+$ and $A^{(t)}_-$ such that
		$$
			[i^t, \x^{(t)}] = A^{(t)}_+(\xi, \tilde{\U}^{\ih^t}, \ih^t) \mathds{1}_{G^{<t}} + A^{(t)}_-(\xi, \bar{\U}^{\ih^t}, \ih^t) \mathds{1}_{(G^{<t})^c}.
		$$
	\end{lemma}
	\begin{proof}{of Lemma \ref{appendix:sublemma}}
		Recall the definition $f^*_i(\x) = \hat{f}_{K;\B_i}(\C_i^T\x) =  \bar{f}_K(\B_i^T\rho(\C_i^T\x)) + \frac{1}{10}\norm{\C_i^T\x}^2$ for convenience. The sequence $\{[i^t,\x^{(t)}]\}_{t\in \N}$ is informed by ${F}^*$.
		Therefore, for any $t \in \N$, there exists a measurable mapping $A^{(t)}$ such that:
		$$
		[i^t, \x^{(t)}] = A^{(t)} 
		\Big \{
		\xi, \ih^t, 
		\x^{(0:t-1)},
		\nabla^{(0:q)}{f}^*_{i^0}(\x^{(0)}), \ldots ,\nabla^{(0:q)}{f}^*_{i^{t-1}}(\x^{(t-1)})
		\Big\}.
		$$
		We show our result by induction on $t \in [T]$. The base case is clear, as the first iterate is $\x^{(0)} = \mathbf{0}$. For the step, assume $G^{<t+1}$ happens and that the result holds for any $s\leq t$. By the derivation on the previous pages we have $\abs{\langle \bb_{i^t,j},\y_{i^t}^{(t)} \rangle} < 1/2$ for all $j \geq I^{t-1}_{i^t} + 1 = I^{t}_{i^t}$.
		Then because $\bar{f}_K$ is a robust zero-chain and $\C$ is fixed, $\nabla^{(0:q)}{f}^*_i(\x^{(t)})$ only depends on $\x^{(t)}$ and columns of $\B_{i^t}$ with indices up to $\min(K,I_{i^t}^{t})$. Note that $\tilde{\U}^{\ih^{t+1}}$ contains all of those columns of $\B_{i^t}$. Therefore the computation of the pair $[i^t,\x^{(t)}]$ only depends on $\x^{(0)},\ldots,\x^{(t)}$, $\ih^{t+1}$ and $\tilde{\U}^{\ih^{t+1}}$ in case $G^{<t+1}$ happens. In that case, we may write
		$$
			[i^{t+1}, \x^{(t+1)}] = A^{(t+1)}_+(\xi, \tilde{\U}^{\ih^{t+1}}, \ih^{t+1}),
		$$
		with the dependence on the previous iterates being implicit (justified by the induction hypothesis).
		This leads to the statement of this sub-lemma.
	\end{proof}
	
	For $t \in [T]$, condition on $G^{<t}$, $\ih^t= \ih_0^t$, $\xi = \xi_0$ and $\tilde{\U}^{\ih_0^t} = \tilde{\U}_0$. Consequently, the iterates $\x^{(1)},\ldots,\x^{(t)}$ are deterministic and so are the $\y_i$'s. Thus for all $i \in [n]$, the quantity $\P_i^{(t-1)\bot}\y_i^{(t)}$ is deterministic as well (recall the definition of $\mathcal{V}_i^{(t-1)})$. 

	For any (still random) $\u \in \mathcal{U}_i^{(t)}$, we are interested in (recall \eqref{eq:appendix:prob:union}):
		\begin{align*}
	&\Pr \left( \abs*{\langle \u, \P_i^{(t-1)\bot}\y_i^{(t)} \rangle} \geq a \norm{\P_i^{(t-1)\bot}\y_i^{(t)}}  \, \vert \, G^{<t}, \ih^t=\ih_0^t,\xi=\xi_0,\tilde{\U}^{\ih_0^t} = \tilde{\U}_0 \right) \\
	\leq \,\, & \Pr \left( \abs*{\left \langle\frac{ \P_i^{(t-1)\bot}\u}{\norm{ \P_i^{(t-1)\bot}\u}}, \frac{\P_i^{(t-1)\bot}\y_i^{(t)}}{\norm{\P_i^{(t-1)\bot}\y_i^{(t)}}}\right \rangle} \geq a   \, \vert \,  G^{<t},\ih^t=\ih_0^t,\xi=\xi_0,\tilde{\U}^{\ih_0^t} = \tilde{\U}_0 \right).
	\end{align*}
	The inequality follows because $\norm{ \P_i^{(t-1)\bot}\u} \leq \norm{\u}$, which holds as $\P_i^{(t-1)\bot}$ is an orthogonal projector.
	By the previous discussion, we know the second term in this scalar product is a deterministic unit vector in the space orthogonal to $\mathcal{V}_i^{(t-1)}$ \footnote{This set is also deterministic as a consequence of the conditioned variables.}.   What remains to study is the distribution of
	$\frac{ \P_i^{(t-1)\bot}\u}{\norm{ \P_i^{(t-1)\bot}\u}}$. We wish to show that $\frac{ \P_i^{(t-1)\bot}\u}{\norm{ \P_i^{(t-1)\bot}\u}}$ is a uniformly distributed unit vector in the space orthogonal to $\mathcal{V}_i^{(t-1)}$. Let $\Z_i \in \R^{d/n\times d/n}$ be a rotation that lets the span of $\mathcal{V}_i^{(t-1)}$ invariant, i.e. $\Z_i ^T\u = \u = \Z_i\u$ for any $\u \in \mathcal{V}_i^{(t-1)}$.
	For a random variable $X$, let $p_{X}$ denote its density. We want to show the equality:
	\begin{align*}
		&\,\,p_{\U^{\ih_0^t}}(\U_0 \, \vert \, G^{<t}, \ih^t = \ih^t_0,\xi = \xi_0,\tilde{\U}^{\ih_0^t} = \tilde{\U}_0 ) \\ = &\,\,p_{\U^{\ih_0^t}}(\Z_i\U_0 \, \vert \,G^{<t}, \ih^t = \ih^t_0,  \xi = \xi_0,\tilde{\U}^{\ih_0^t} = \tilde{\U}_0 ),
	\end{align*}
	to show the distribution of $\frac{ \P_i^{(t-1)\bot}\u}{\norm{ \P_i^{(t-1)\bot}\u}}$ is indeed uniform.
	
	Let $\bar{\U}_0 = [\tilde{\U}_0, \U_0]$. We lighten the notation up a bit by omitting the random variables where they are clear from context. Using conditional densities:
	\begin{align*}
		p_{\U^{\ih_0^t}}(\U_0 \, \vert \, \ih^t_0,  G^{<t},\xi_0,\tilde{\U}_0 ) 
		&=\frac{\Pr( G^{<t}, \ih^t = \ih_0^t \, \vert \, \xi_0, \U_0, \tilde{\U}_0) p_{\xi,\bar{\U}^{\ih^t_0}}(\xi_0, \bar{\U}_0)}{\Pr(G^{<t}, \ih^t = \ih_0^t \, \vert \,\xi_0, \tilde{\U}_0 ) p_{\xi,\tilde{\U}^{\ih_0^t}}(\xi_0, \tilde{\U}_0)}
		\\ &= \frac{\Pr( G^{<t}, \ih^t = \ih_0^t \, \vert \, \xi_0, \bar{\U}_0) p_{\bar{\U}^{\ih^t_0}}( \bar{\U}_0)}{\Pr(G^{<t} ,\ih^t = \ih_0^t \, \vert \,\xi_0, \tilde{\U}_0 ) p_{\tilde{\U}^{\ih_0^t}}(\tilde{\U}_0)}.
	\end{align*}
	Plugging in $\Z_i \U_0$ and using $\Z_i \tilde{\U}_0 = \tilde{\U}_0$ we obtain
	\begin{align*}
		p_{\U^{\ih_0^t}}(\Z_i\U_0 \, \vert \, G^{<t}, \ih^t_0, \xi_0,\tilde{\U}_0 ) 
	&=\frac{\Pr( G^{<t}, \ih^t = \ih_0^t \, \vert \, \xi_0, \Z_i \bar{\U}_0) p_{\bar{\U}^{\ih^t_0}}( \Z_i\bar{\U}_0)}{\Pr(G^{<t}, \ih^t = \ih_0^t \, \vert \,\xi_0, \tilde{\U}_0 ) p_{\tilde{\U}^{\ih_0^t}}( \tilde{\U}_0)}.
	\end{align*}
	Because of the uniform distribution of $\B$ and thus also of $\bar{\U}^{\ih^t_0}$, it suffices to show that
	$$
		\Pr( G^{<t}, \ih^t = \ih_0^t \, \vert \, \xi_0, \bar{\U}_0)  = \Pr( G^{<t}, \ih^t = \ih_0^t \, \vert \, \xi_0, \Z_i\bar{\U}_0).
	$$	
	This probability is either 0 or 1, because we condition on all the randomness involved. We show by induction on $s \in [t]$ that $\Pr( G^{<t}, \ih^t = \ih_0^t \, \vert \, \xi_0, \bar{\U}_0) =1$ implies $\Pr( G^{<t}, \ih^s = \ih_0^s \, \vert \, \xi_0, \Z_i\bar{\U}_0) =1$. The other direction is analogous. 
	
	Therefore assume $\ih^t = \ih_0^t$ and that $ G^{<t}$ happens, conditioned on $\xi=\xi_0$ and $\bar{\U}^{\ih^t_0} = \bar{\U}_0$. 
	The base case is trivial, because $G^{<1}$ always happens. For the inductive step, let $s \geq 2$ and
	assume that $\ih^{s-1} = \ih_0^{s-1}$ and $ G^{<s-1}$ happen, conditioned on $\xi=\xi_0$ and $\bar{\U}^{\ih^t_0} = \Z_i\bar{\U}_0$ (induction hypothesis).  
	
	Let $i'^{s-1}$, $\x'^{(s-1)}$ denote the next index and iterate the algorithm produces, given $\bar{\U}^{\ih^t_0} = \Z_i\bar{\U}_0$.  By Lemma \ref{appendix:sublemma}, the induction hypothesis allows us to write for some $A_+^{(s-1)}$
	\begin{align}
		[i'^{s-1},\x'^{(s-1)}] &= A^{(s-1)}_+(\xi_0, \Z_i\tilde{\U}_0, \ih_0^{s-1})\nonumber \\
		&= A^{(s-1)}_+(\xi_0, \tilde{\U}_0, \ih_0^{s-1}) \nonumber \\
		&= [i^{s-1}, \x^{(s-1)}],		\label{eq:usefulequality_appendix}
		\end{align}
		where we also used $\Z_i\tilde{\U}_0 = \tilde{\U}_0$.
	This means that $\ih^s = \ih^s_0$ iff $\ih'^s = \ih^s_0$, which gets us halfway there. We just have to show that $G^{<s}$ happens as well, given $\bar{\U}^{\ih^t_0} = \Z_i\bar{\U}_0$. Of course, showing $G^{s-1}$ suffices, by the induction hypothesis. For this, let $\u \in \mathcal{U}^{(s-1)}$ and $i\in [n]$. We have
	\begin{align*}
	\left\langle \Z_i\u, \frac{\P_i^{(s-2)\bot} \y_i'^{(s-1)}}{\norm{\P_i^{(s-2)\bot} \y_i'^{(s-1)}}}\right\rangle 
	&= \left\langle \u, \Z_i^T\frac{\P_i^{(s-2)\bot} \y_i^{(s-1)}}{\norm{\P_i^{(s-2)\bot} \y_i^{(s-1)}}}\right\rangle \\
	&= \left\langle \u, \frac{\P_i^{(s-2)\bot} \y_i^{(s-1)}}{\norm{\P_i^{(s-2)\bot} \y_i^{(s-1)}}}\right\rangle.
	\end{align*}
	The first equality follows because $\P_i^{(s-2)\bot} \y_i'^{(s-1)} = \P_i^{(s-2)\bot} \y_i^{(s-1)}$ by
	\eqref{eq:usefulequality_appendix} and the second step follows because $\P_i^{(s-2)\bot} \y_i^{(s-1)} = \y_i^{(s-1)} - \P_i^{(s-2)} \y_i^{(s-1)}$ is in the span of $\mathcal{V}_i^{(s-1)} \subset \mathcal{V}_i^{(t)}$ and  left invariant by $\Z_i^T$. Thus $G^{s-1}$ holds as well, conditioned on $\bar{\U}^{\ih^t_0} = \Z_i\bar{\U}_0$.  
	
	 This concludes the inductive step and therefore our proof that 
	$\frac{ \P_i^{(t-1)\bot}\u}{\norm{ \P_i^{(t-1)\bot}\u}}$ is a uniformly distributed unit vector in a subspace of $\R^{d/n}$ of dimension at least 
	$$d' \geq d/n - \abs{\mathcal{V}_i^{(t-1)}} \geq d/n - (t-1) - \sum_{j=1}^{n}\min(I_j^{(t-1)},K) \geq d/n - 2t.$$ 
	We may write our probability to bound 
	\begin{align*}
		\Pr \left( \abs*{\left \langle\frac{ \P_i^{(t-1)\bot}\u}{\norm{ \P_i^{(t-1)\bot}\u}}, \frac{\P_i^{(t-1)\bot}\y_i^{(t)}}{\norm{\P_i^{(t-1)\bot}\y_i^{(t)}}}\right \rangle} \geq a   \, \vert \, G^{<t}, \ih^t=\ih_0^t,\xi=\xi_0,\tilde{\U}^{\ih_0^t} = \tilde{\U}_0 \right)
	\end{align*}
	as 
	$$\Pr(\abs{v_1} \geq a),$$
	where $\vvv$ is a uniformly distributed unit vector in $\R^{d'}$. This is because for the dot product, only the angle between the two vectors matters and with all conditioned variables, $\frac{\P_i^{(t-1)\bot}\y_i^{(t)}}{\norm{\P_i^{(t-1)\bot}\y_i^{(t)}}} $ is fixed so we may assume w.l.o.g. that it is equal to $\mathbf{e}_1$. By a standard concentration of measure bound on the sphere (see Lecture 8 in \citet{ball:concentration}) we get
	$$
		\Pr(\abs{v_1} \geq a)  = \Pr(\abs{v_1} > a) \leq 2e^{-d'a^2/2} \leq 2e^{-\frac{a^2}{2}(d/n- 2t)} \leq 2e^{-\frac{a^2}{2}(d/n- 2T)} .
	$$
	Returning to \eqref{eq:appendix:prob:union} we get for all $t \in [T]$ a bound for \eqref{eq:app:thingtobound} of
	\begin{align*}
		\Pr((G^{\leq t})^c  \, \vert \, G^{<t}) 
		\leq n \cdot nK \cdot 2e^{-\frac{a^2}{2}(d/n- 2T)},
	\end{align*}
	and therefore by \eqref{eq:app:motherofunions}
		\begin{align*}
	\Pr((G^{\leq T})^c  ) 
	\leq T \cdot n \cdot nK \cdot 2e^{-\frac{a^2}{2}(d/n- 2T)} 
	\leq 2(nK)(n^2K)e^{-\frac{a^2}{2}(d/n- 2T)}.
	\end{align*}
	Setting 
	$$
		d/n \geq \frac{2}{a^2} \log\left(\frac{2n^3K^2}{\delta}\right) + 2T
	$$
	gives us a probability $\delta$ bound. By the definitions of $a$ and $T$, the choice
	$$
		d/n \geq 2\max(9n^2K^2, 12nKR^2)\log\left(\frac{2n^3K^2}{\delta}\right) + nK
	$$
	suffices. This concludes the proof.
\end{proof}

The following Lemma is related to \citet{carmon:lower:i}, Lemma 5.

\begin{proof}{of Lemma \lemmagradbound{}}
	Fix $t \in \{0,\ldots,T\}$. For any $i \in [n]$, define $\y_i^{(t)} = \rho(\C_i^T \x^{(t)})$. Then Lemma~\lemmakey{} gives 
	for all $\u \in \mathcal{U}_i^{(t)}$ that $\abs{\langle \u, \y_i^{(t)} \rangle} < 1/2$. Therefore for each $i$ with $\mathcal{U}_i^{(t)} \not = \emptyset$ we have some $k \in [K]$ with 
	$$
		 \abs{\langle \bb_{i,k}, \y_i^{(t)} \rangle} < \frac{1}{2} < 1.
	$$
	With $\z = \B_i ^T\y_i^{(t)}$, by Lemma~\lemoriginallargegrad{} there exists an index $j \leq k$ with $ \abs{z_j} = \abs{\langle \bb_{i,j}, \y_i^{(t)} \rangle} < 1$ and
	$$	
		\abs*{\frac{\partial \bar{f}_K}{\partial z_j}(\B_i^T \y^{(t)}) } = \abs*{\frac{\partial \bar{f}_K}{\partial z_j}(\z) } > 1.
	$$
	Define $\tilde{f}_{K;B_i}(\y_i^{(t)}) = \bar{f}_K(\B_i^T\y_i^{(t)})$
	and recall the definitions of $\bar{f}_K$ and $\hat{f}_{K;\B_i}$. They give
	$$
	 \hat{f}_{K;\B_i}(\C_i^T\x) =  \tilde{f}_{K;\B_i}(\rho(\C_i^T\x)) + \frac{1}{10}\norm{\C_i^T\x}^2  = \bar{f}_K(\B_i^T\rho(\C_i^T\x)) + \frac{1}{10}\norm{\C_i^T\x}^2 .
	$$
	By the chain rule we have
	$$
		\B_i^T (\nabla \tilde{f}_{K;\B_i} (\y_i^{(t)})) = \B_i^T(\B_i\nabla \bar{f}_K(\B_i^T\y_i^{(t)})) = \nabla \bar{f}_K(\B_i^T\y_i^{(t)}).
	$$
	Combining this with the above we deduce that
	$$
		\abs*{\langle \bb_{i,j}, \nabla \tilde{f}_{K;\B_i}(\y_i^{(t)}) \rangle} = \abs*{\frac{\partial \bar{f}_K}{\partial z_j}(\B_i^T \y^{(t)}) } > 1.
	$$
	\citet{carmon:lower:i} show  that $\abs{\langle \bb_{i,j}, \y_i^{(t)} \rangle} < 1$ and $\abs*{\langle \bb_{i,j}, \nabla \tilde{f}_{K;\B_i}(\y_i^{(t)}) \rangle} > 1$ imply
	$$
		\norm{\nabla \hat{f}_{K;\B_i}(\C_i^T\x^{(t)})} > \frac{1}{{2}},
	$$
	where the gradient is w.r.t. the function argument, i.e. $\C_i^T \x^{(t)}$. They show this in the proof of Lemma 5, in the calculations following Equation (14) \footnote{With slightly different naming. Replace $U$ with $\B_i$ and $u^{(j)}$ with $\bb_{i,j}$, $T$ with $K$ and $x^{(t)}$ with $\C_i^T \x^{(t)}$. Also note that this is the part where the added regularization term in $\hat{f}$ is needed.}.
	
	The only thing that remains to show is that this indeed guarantees $\nabla F^*(\x^{(t)})$ to be large. Note that in each iteration, one of the $\mathcal{U}_i$'s shrinks in size by at most 1, while the others do not change. That means that after $t \leq T = \frac{nK}{2}$ iterations, at most $\lfloor n/2 \rfloor$ indices $i$ can have $\mathcal{U}_i^{(t)} = \emptyset$. Let $J \subset [K]$ be the set of those indices $i$ with $\mathcal{U}_i^{(t)} \not= \emptyset$. Then $\abs{J} \geq n/2$ and
	\begin{align*}
		\norm{\nabla F^*(\x^{(t)})}^2 &=  \norm{\frac{1}{n}\sum_{i=1}^n \nabla f^*_i(\x^{(t)})}^2 \\
		&=  \norm{\frac{1}{n}\sum_{i=1}^n \nabla \left[\hat{f}_{K;\B_i}   \left(\C_i^T\x^{(t)} \right) \right]}^2 \\
		&=  \norm{\frac{1}{n} \sum_{i=1}^n \C_i \nabla \hat{f}_{K;\B_i} \left(\C_i^T\x^{(t)} \right) }^2 \\	
		&= \frac{1}{n^2}\sum_{i=1}^n\sum_{j=1}^n \left(\nabla \hat{f}_{K;\B_i} \left(\C_i^T\x^{(t)} \right) \right)^T \C_i^T\C_j \nabla \hat{f}_{K;\B_j} \left({\C_j^T\x^{(t)}} \right) \\
		&\stackrel{(*)}{=} \frac{1}{n^2}\sum_{i=1}^n \left(\nabla \hat{f}_{K;\B_i} \left(\C_i^T\x^{(t)} \right) \right)^T \C_i^T\C_i \nabla \hat{f}_{K;\B_i} \left({\C_i^T\x^{(t)}} \right) \\
		&= \frac{1}{n^2}\sum_{i=1}^n \norm{\nabla \hat{f}_{K;\B_i} \left(\C_i^T\x^{(t)} \right)  }^2 \\
		&\geq \frac{1}{n^2}  \sum_{i\in J}\norm{\nabla \hat{f}_{K;\B_i} \left(\C_i^T\x^{(t)} \right)  }^2 \\
		&\geq \frac{1}{n^2}  \frac{n}{2}\frac{1}{4} \\
		&\geq \frac{1}{16n}.
	\end{align*}
	where $(*)$ is because of the definition of $\C \in \ortho(d,d)$.
\end{proof}

\renewcommand{\curir}{\x_t^s} 
\renewcommand{\estim}{\v_t^s} 
\renewcommand{\hestim}{\U_t^s} 
\renewcommand{\snap}{\widehat{\x}^s} 
\renewcommand{\snapgrad}{\mathbf{g}^s}
\renewcommand{\snaphess}{\mathbf{H}^s}
\renewcommand{\step}{\mathbf{h}_t^s}

\renewcommand{\grdf}{\nabla F} 
\renewcommand{\hesf}{\nabla^2 F} 
\renewcommand{\grdfc}{\grdf(\curir)} 
\renewcommand{\hesfc}{\hesf(\curir)} 
\renewcommand{\grdfi}{\nabla f_{i_t}} 
\renewcommand{\hesfi}{\nabla^2 f_{i_t}} 
\renewcommand{\grdfj}{\nabla f_{i_t}} 
\renewcommand{\hesfj}{\nabla^2 f_{j_t}} 

\renewcommand{\xdiff}{\curir - \snap} 
\newcommand{\xdiffnorm}{\norm{\xdiff}} 
\newcommand{\ycut}{\mathbf{Y}_{j_t}} 
\newcommand{\stepnorm}{\norm{\step}}
\newcommand{\grddiff}{\grdfc - \estim} 
\newcommand{\hesdiff}{\hesfc - \hestim} 
\newcommand{\grddiffnorm}{\norm{\grddiff}} 
\newcommand{\hesdiffnorm}{\norm{\hesdiff}} 
\newcommand{\fidiffnesres}{\grdfi(\curir) - \grdfi(\snap) - \hesfi(\snap)(\curir - \snap)} 
\newcommand{\fdiffnesres}{\grdf(\curir) - \grdf(\snap) - \hesf(\snap)(\curir - \snap)} 

\newcommand{\Ei}{\E_{i_t}}
\newcommand{\Ej}{\E_{j_t}}

\newcommand{\lammin}{\lambda_\mathrm{min}}
\newcommand{\xout}{\x_{\mathrm{out}}}
\newcommand{\lamhes}{\lammin(\hesf(\x_{\mathrm{out}}))}

\section{Proof of results under Assumption \assthird{}}
\label{appendix:C}
The convergence analysis in this section closely follows \citet{zhou:svrc}, many parts of which are be left unchanged. We argue that this supports our claim that Assumption~\ref{assumption:third} is a natural smoothness assumption.

\subsection{Proof of Theorem \theothirdup{}}
Recall the terminology in Algorithm~\ref{algorithm:svrc}. We will commonly call $\estim$ and $\hestim$ the gradient and Hessian estimators respectively, we will refer to $\snap$ as the snapshot point, and to $\step$ as the step. Finally, we will define
$$
m_t^s(\mathbf{h}) = \inner{\estim}{\mathbf{h}} + \frac{1}{2}\inner{\hestim \mathbf{h}}{ \mathbf{h}} + \frac{M}{6}\norm{\mathbf{h}}^3,
$$
so that $\step = \arg\min_{\mathbf{h}}m_t^s(\mathbf{h})$.
To aid in the analysis, we define the following quantity also introduced in \citet{zhou:svrc}:
$$
\mu(\x) = \max \left\{ \norm{\grdf(\x)}^{3/2}, -\frac{\lambda_\mathrm{min}^3(\hesf(\x))}{L_{2}^{3/2}}\right\}.
$$
Whenever $\mu(\x) \leq \epsilon^{3/2}$, $\x$ is an $\epsilon$-approximate local minimum \cite{zhou:svrc}. In Section~\ref{sec:svrc_theo_proof}, we will show that we can bound the expected value of this quantity as follows (see also Theorem 6 in \citet{zhou:svrc}):
\begin{theorem}
\label{svrc:theorem}
Let $M = C_M L_2$ for $C_M = 150$. Let $T \geq 2$ and choose $b_g \geq 5T^4$ and $b_h \geq 3000T^2\log^3 d$. Then 
$$
    \E[\mu(\x_{\mathrm{out}})] \leq \frac{240 C_M^2L_2^{1/2}\Delta}{ST}.
$$
\end{theorem}
Using this, we proceed with the proof of the main upper bound result.

\begin{proof}{of Theorem \theothirdup{}}
We first check that in the setting of Theorem \theothirdup{}, the assumptions of Theorem~\ref{svrc:theorem} hold. It is clear that $T \geq 2$ and that $b_g = 5 \max\{n^{4/5}, 2^4\} = 5T^4$. Further, $b_h = 3000 \max\{4, n^{2/5}\} \log^3 d = 3000 T^2 \log^3 d$. Plugging in the choices of $S$ and $T$ into the result of Theorem~\ref{svrc:theorem}, one gets
$$
    \E[\mu(\x_{\mathrm{out}})] \leq \frac{240 C_M^2L_2^{1/2}\Delta}{ST} \leq \frac{240 C_M^2L_2^{1/2}\Delta}{\max\{ 1, 240 C_M^2 L_2^{1/2} \Delta n^{-1/5}\epsilon^{-3/2}\} \max\{2,n^{1/5}\}} \leq \epsilon^{3/2},
$$ as desired.
In particular, we have
$$
    \E [\norm{\grdf(\x_{\mathrm{out}})}]^{3/2} \leq \E [\norm{\grdf(\x_{\mathrm{out}})} ^{3/2} ]\leq \epsilon^{3/2},
$$
allowing comparison with our lower bound from Theorem~\theothirdlow{}.

During each epoch, $n$ oracle calls are needed to construct $\snapgrad$ and $\snaphess$, requiring $Sn$ calls overall. To compute $\estim$ and $\hestim$, we need 
$$
 b_g + b_h = 5 \max\{n^{4/5}, 2^4\} + 3000 \max\{4, n^{2/5}\} \log^3
$$
oracle queries at each iteration, requiring $ST(b_g + b_h)$ calls over all epochs and iterations. The total number of oracle queries is therefore at most
\newcommand{\const}{240 C_M^2L_2^{1/2}\Delta}
\begin{align*}
    & \quad\,\, Sn + ST(b_g + b_h)  \\
    &= \max\{ 1, 240 C_M^2 L_2^{1/2} \Delta n^{-1/5}\epsilon^{-3/2}\} n \\ & \quad \quad
    + (\max\{ 1, 240 C_M^2 L_2^{1/2} \Delta n^{-1/5}\epsilon^{-3/2}\})(\max\{2,n^{1/5}\})(5 \max\{n^{4/5}, 2^4\} + 3000 \max\{4, n^{2/5}\} \log^3 d) \\
    &\leq \tilde{\upp} \left( n+\frac{\Delta L_2^{1/2} n^{4/5}}{\epsilon^{3/2}} \right).
\end{align*}
\end{proof}

\subsection{Proof of Theorem \ref{svrc:theorem}}
\label{sec:svrc_theo_proof}
We will need some auxiliary lemmas to conduct the proof. The first is a version of Lemma 1 from \citet{nesterov:cr}, but tailored to our finite-sum setting.

\begin{lemma}
\label{lemma:modnesterov}
Let $F = \frac{1}{n}\sum f_i$ satisfy Assumption~\assthird{}. Then we have for any $\x$ and $\y$\emph{:}
$$
\E_i\left[\norm{\nabla f_i(\y) - \nabla f_i(\x) - \nabla^2 f_i(\x)(\y - \x)}^2 \right] \leq  \frac{1}{3}L_2^2 \norm{\x-\y}^4.
$$
and for any $\mathbf{h}$\emph{:}
$$
F(\x+\mathbf{h}) \leq F(\x) + \inner{\grdf(\x)}{\mathbf{h}}+\frac{1}{2}\inner{\hesf(\x)\mathbf{h}}{\mathbf{h}}+\frac{L_2}{6}\norm{\mathbf{h}}^3.
$$
\end{lemma}
The second statement is taken directly from \citet{nesterov:cr}.
We also take the following lemma directly from \citet{zhou:svrc}. Its proof exploits the optimality of $\step$.
\begin{lemma}[Lemma 24 in \citet{zhou:svrc}]
\label{lemma:svrc24}
For the iterates in Algorithm~\ref{algorithm:svrc} under the assumptions of Theorem~\theothirdup{} we have 
\begin{align*}
\estim + \hestim \step + \frac{M}{2}\stepnorm \step &= 0, \\
\hestim + \frac{M}{2}\stepnorm \mathbf{I} \succeq 0, \\
\inner{\estim}{\step} + \frac{1}{2}\inner{\hestim\step}{\step} + \frac{M}{6}\stepnorm^3 \leq -\frac{M}{12}\stepnorm^3.
\end{align*}
\end{lemma}

The two following lemmas resemble Lemmas 25 and 26 in \citet{gu:lower} and bound the variances of the gradient and Hessian estimators of SVRC. Under the new smoothness Assumption~\assthird{}, some constant factors change and the batch size for the Hessian estimator must comply to some stronger requirements, but otherwise third-moment smoothness is a viable alternative to an individual smoothness assumption.
The proofs are analogous to the proofs of their respective counterparts.

The first lemma bounds the variance of $\estim$:
\begin{lemma}
\label{lemma:gradientestimator}
The gradient estimator $\estim$ in Algorithm~\ref{algorithm:svrc} satisfies
$$
    \Ei \norm{\nabla F(\curir) - \estim}^{3/2} \leq \frac{2L_2^{3/2}}{b_g^{3/4}}\norm{\curir - \snap}^3,
$$
where $\Ei$ is the expectation over the batch indices $i_t \in I_g$.
\end{lemma}

The second lemma in this section bounds the variance of $\hestim$:
\begin{lemma}
\label{lemma:hessianestimator}
If $b_h \geq 12000\log^3 d$, the Hessian estimator $\hestim$ satisfies
$$
\Ej \hesdiffnorm^3 \leq 15000L_2^3\left(\frac{\log d}{b_h}\right)^{3/2}\xdiffnorm^3,
$$
where $\Ej$ is the expectation over the batch indices $j_t \in I_h$.
\end{lemma}

For completeness, we provide the rest of the lemmas from \citet{zhou:svrc} that are needed in the analysis. We change the wording a bit, to make their applicability explicit, but all the proofs in the original paper can be applied \emph{unchanged}, as is easily checked.

Lemma~\ref{lemma:svrc27} can be derived using the Cauchy-Schwarz and Young inequalities.
\newcommand{\h}[0]{\mathbf{h}}
\begin{lemma}[Lemma 27 in \citet{zhou:svrc}]
\label{lemma:svrc27} For the iterates in Algorithm~\ref{algorithm:svrc} under the assumptions of Theorem~\theothirdup{} and for any $\h$, we have
\begin{align*}
    \inner{\grdfc - \estim}{\h} &\leq \frac{M}{27}\norm{\h}^3 + \frac{2\grddiffnorm^{3/2}}{M^{1/2}}, \\
    \inner{\hesdiff}{\h} &\leq \frac{2M}{27}\norm{\h}^3 + \frac{27}{M^2}
\hesdiffnorm^3.
\end{align*}
\end{lemma}

\begin{lemma}[Lemma 28 in \citet{zhou:svrc}] For the iterates in Algorithm~\ref{algorithm:svrc} under the assumptions of Theorem~\theothirdup{} and 
\label{lemma:svrc28} for any $\h$, we have
\begin{align*}
    \mu(\curir + \h) &\leq 9C_M^{3/2}\Big[M^{3/2}\norm{\mathbf{h}}^3 + \grddiffnorm^{3/2}+ M^{-3/2}\hesdiffnorm^3 \\ & \quad + \norm{\nabla m_t^s (\h)}^{3/2} + M^{3/2} \big\lvert {\norm{\h}-\stepnorm}\big \rvert ^3\Big].
\end{align*}
\end{lemma}

\begin{lemma}[Lemma 29 in \citet{zhou:svrc}]
\label{lemma:svrc29} 
For any $\x,\y,\h$ and $C \geq 3/2$ we have
\begin{align*}
    \norm{\x + \h - \y}^3 \leq 2C^2 \norm{\h}^3 + (1+3/C)\norm{\x-\y}^3.
\end{align*}
\end{lemma}

\begin{lemma}[Lemma 30 in \citet{zhou:svrc}]
\label{lemma:svrc30} Define $c_T = 0$ and for $t \in [0:T-1]$ define $c_t = c_{t+1}(1+3/T)+M(500T^3)^{-1}$. Then for any $t \in [1:T]$ we have:
$$
M/24 - 2c_tT^2 \geq 0.
$$
\end{lemma}

\begin{proof}{of Theorem \ref{svrc:theorem}}
This proof is very close to identical to the one of Theorem 6 in \citet{zhou:svrc}, but we give it again for completeness, with the changes coming from the slightly modified lemmas. We can bound the function value at the next iterate $F(\x_{t+1})$ as follows:
\begin{align}
F(\x_{t+1}^s) &\leq F(\curir) + \inner{\grdfc}{\step} + \frac{1}{2}\inner{\hesfc\step}{\step} + \frac{L_2}{6}\norm{\step}^3  \label{eq:theoremineq1}\\
&= F(\x_t^s) + \inner{\estim}{\step} + \frac{1}{2}\inner{\hestim\step}{\step} + \frac{M}{6}\stepnorm^3 + \inner{\grddiff}{\step} \nonumber \\ 
& \quad \quad + \frac{1}{2}\inner{\left( \hesdiff \right)\step}{\step} + \frac{M - L_2}{6}\stepnorm^3 \nonumber \\
&\leq F(\curir) - \frac{M}{2}\stepnorm^3 + \left( \frac{M}{27} \stepnorm^3 + \frac{2\grddiffnorm^{3/2}}{M^{1/2}}\right) \nonumber \\
& \quad \quad + \frac{1}{2}\left( \frac{2M}{27}\stepnorm^3 + \frac{27}{M^2}\hesdiffnorm^3\right) - \frac{M - L_2}{6}\stepnorm^3 \label{eq:theoremineq2} \\
&\leq F(\curir) - \frac{M}{12}\stepnorm^3 + \frac{2}{M^{1/2}}\grddiffnorm^{3/2}+\frac{27}{M^2}\hesdiffnorm^3 \label{eq:theorem:nextbound}.
\end{align}
\eqref{eq:theoremineq1} holds due to Lemma~\ref{lemma:modnesterov} and \eqref{eq:theoremineq2} is valid because of Lemmas~\ref{lemma:svrc24} and \ref{lemma:svrc27}.

Define
$$ R_t^s = \E \left[F(\curir) + c_t\xdiffnorm^3\right],$$
where $c_T = 0$ and $c_t = c_{t+1}(1+3/T)+ M(500T^3)^{-1}$ for $t \in [0:T-1]$. We use Lemma~\ref{lemma:svrc29} with $T \geq 2 \geq 3/2$ to get a recurrence -- involving the step -- for the cubed distance from an iterate to the snapshot point:
\begin{equation}
c_{t+1}\norm{\x_{t+1}^s- \snap}^3 \leq 2c_{t+1}T^2\norm{\step}^3+c_{t+1}(1+3/T)\norm{\curir - \snap}^3 \label{eq:theoremproof:sequence}.
\end{equation}

We can make use of Lemma \ref{lemma:svrc28} with $\mathbf{h} = \step$ followed by Lemma \ref{lemma:svrc24}
\begin{align}
(240C_M^2L_2^{1/2})^{-1}\mu(\x_{t+1}^s) 
&\leq \frac{M}{24}\stepnorm^3+\frac{\grddiffnorm^{3/2}}{24M^{1/2}}+\frac{\hesdiffnorm^3}{24M^2} \nonumber \\ 
&\quad \quad + \frac{\norm{\nabla m_t^s(\step)}^{3/2}}{24M^{1/2}} + \frac{M}{24}\big \lvert \stepnorm - \stepnorm \big \rvert^3 \nonumber \\ 
&= \frac{M}{24}\stepnorm^3+\frac{\grddiffnorm^{3/2}}{24M^{1/2}}+\frac{\hesdiffnorm^3}{24M^2}, \label{eq:theoremproof:mubound}.
\end{align}
 In the first step we used $C_M = 150$ and $M = C_M L_2$ and in the second we used the optimality of $\step$ as an argument of $m_t^s$. Our aim is to get a telescoping sum for the $R_t$'s. For that, we start by combining \eqref{eq:theorem:nextbound}, \eqref{eq:theoremproof:sequence} and \eqref{eq:theoremproof:mubound} (this time the expectation is over all the randomness involved in the algorithm):
\begin{align}
    R_{t+1}^s + (240C_M^2L_2^{1/2})^{-1}\E[\mu(\x_{t+1})]  
    &= \E\left[F(\x_{t+1}^s) + c_{t+1}\norm{\x_{t+1}^s- \snap}^3 + (240C_M^2L_2^{1/2})^{-1}\mu(\x_{t+1}^s)\right]\nonumber \\
    &\leq \E\left[F(\curir) + c_{t+1}(1+3/T)\norm{\curir - \snap}^3 - (M/24 - 2c_{t+1}T^2)\norm{\step}^3\right]\nonumber \\ 
    &\quad \quad+ \E\left[3M^{-1/2} \grddiffnorm^{3/2} + 28M^{-2}\hesdiffnorm^3\right] \nonumber\\
    &\leq \E\left[F(\curir) + c_{t+1}(1+3/T)\norm{\curir - \snap}^3\right] \nonumber\\ 
    &\quad \quad+ \E\left[3M^{-1/2} \grddiffnorm^{3/2} + 28M^{-2}\hesdiffnorm^3\right] \label{eq:theorem:startrecurrence},
\end{align}
because by Lemma~\ref{lemma:svrc30} we have $M/24 - 2c_{t+1}T^2 \geq 0$ for any $t \in [T]$.
In the second term of \eqref{eq:theorem:startrecurrence}, we recover the gradient and Hessian estimator variances that Lemmas \ref{lemma:gradientestimator} and \ref{lemma:hessianestimator} control. Indeed, taking iterated expectations yields
$$
    3M^{-1/2} \grddiffnorm^{3/2} \leq \frac{6 L_2^{3/2}}{M^{1/2}b_g^{3/4}}\E \xdiffnorm^3 \leq \frac{M}{1000T^3}\E\xdiffnorm^3.
$$
Here we have used that $M = 150L_2$ and $b_g \geq 5T^4$. For the Hessian estimator, we get
\begin{align*}
    28M^{-2}\hesdiffnorm^3 &\leq \frac{28\cdot 15000 L_2^3}{M^2(b_h/\log d)^{3/2}}\E \xdiffnorm^3 \\&\leq \frac{28\cdot 15000M}{150^3(3000)^{3/2}T^3}\E \xdiffnorm^3 
    \\&\leq \frac{M}{1000T^3}\E\xdiffnorm^3,
\end{align*}
where we additionally use $b_h \geq 3000 T^2 \log^3 d$. Note that our larger $b_h$ actually gives us better constant factors than we derive, but we do not need this and therefore keep the same as in the original proof. From here, we exactly follow said original proof from \citet{zhou:svrc}.
We can plug those 2 bounds back into \eqref{eq:theorem:startrecurrence} and use the definition of $c_t$ to get the recurrence
\begin{align*}
    R_{t+1}^s + (240C_M^2L_2^{1/2})^{-1}\E[\mu(\x_{t+1})]  &\leq \E\left[F(\curir) + \norm{\curir - \snap}^3\left(c_{t+1}(1+3/T) + \frac{M}{500T^3} \right)\right] \\
    &= \E[F(\curir) + c_t\xdiffnorm^3]
    = R_t^s.
\end{align*}
We will now do 2 steps of telescoping. First, let $s \in [S]$ be arbitrary. As $c_T = 0$ and $x_T^s = \widehat{\x}^{s+1}$ by definition, we have $R_T^s = \E[F(\x_T^s) + c_T\norm{\x_T^s - \snap}^3] = \E F(\x_T^s) = \E F(\widehat{\x}^{s+1})$. As $\x_0^s = \snap$, we have $R_0^s = \E[F(\x_0^s) + c_0\norm{\x_0^s-\snap}^3] = \E F(\snap)$. Thus, rearranging and telescoping the above from $t=0$ to $T-1$ yields
$$
\E F(\snap) - \E F(\widehat{\x}^{s+1}) = R_0^s-R_T^s \geq \sum_{t=1}^T(240C_M^2 L_2^{1/2})^{-1}\E[\mu (\curir)].
$$
Further, we can telescope this from $s=1$ to $S$ and obtain
$$
    \Delta \geq F(\widehat{\x}^1) - F(\widehat{\x}^S) = \sum_{s=1}^S\left[\E F(\snap) - \E F(\widehat{\x}^{s+1})\right] \geq 
    (240C_M^2 L_2^{1/2})^{-1}\sum_{s=1}^S\sum_{t=1}^T\E[\mu(\curir)].
$$
The first inequality holds because of the definition of $\x_0 = \widehat{\x}^1$ and because the choice of $\step$ guarantees the iterates do not yield increases in function value over time. Therefore, picking a random iterate $\curir$, we will have 
$$
    \E[\mu(\curir)] \leq \frac{240 C_M^2L_2^{1/2}\Delta}{ST},
$$
as desired.
\end{proof}

\subsection{Proof of technical lemmas for the upper bound}
\begin{proof}{of Lemma \ref{lemma:modnesterov}}
We have 
\begin{align*}
    \E_i\norm{\nabla f_i(\y) - \nabla f_i(\x) - \nabla^2 f_i(\x)(\y - \x)}^2 &= \E_i \norm{\int_0^1[\nabla^2 f_i(\x + \tau(\y - \x)) - \nabla^2 f_i(\y)](\y - \x) d\tau}^2 \\
    &\leq \E_i \int_0^1\norm{\nabla^2 f_i(\x + \tau(\y - \x)) - \nabla^2 f_i(\y)}^2\norm{\x - \y}^2 d\tau \\
    &= \int_0^1 \E_i\norm{\nabla^2 f_i(\x + \tau(\y - \x)) - \nabla^2 f_i(\y)}^2\norm{\x - \y}^2 d\tau \\
    &\leq \int_0^1 L_2^2\norm{\x + \tau(\y - \x) - \y}^2 \norm{\x - \y}^2 d\tau \\
    &= \frac{L_2^2}{3}\norm{\x - \y}^4,
\end{align*}
where the first inequality is because of $\norm{\int_0^1 \v d\tau}^2 \leq \left(\int_0^1 \norm{ \v} d\tau\right)^2 \leq \int_0^1 \norm{ \v}^2 d\tau$ and the second inequality follows because of Assumption~\assthird{} and $\E[\abs{X}^s]^{1/s}\leq \E[\abs{X}^t]^{1/t}$ for $s \leq t$.
\end{proof}

To prove Lemma \ref{lemma:gradientestimator} we will need the following technical result:

\begin{lemma}[Lemma 31 in \citet{zhou:svrc}]
\label{lemma:svrc31} Suppose $\mathbf{a}_1,\ldots,\mathbf{a}_N$ are i.i.d. and $\E\mathbf{a}_i = 0$ for all $i$. Then
$$
\E \norm{\frac{1}{N}\sum_{i=1}^N\mathbf{a}_i}^{3/2}\leq \frac{1}{N^{3/4}}(\E\norm{\mathbf{a}_i}^2)^{3/4}.
$$
\end{lemma}

\begin{proof}{of Lemma~\ref{lemma:gradientestimator}}
Using the definition of $\estim$, we can write
\begin{align*}
    &\quad \,\,\Ei \norm{\grdf (\curir) - \estim}^{3/2} \\
    &=\Ei \norm{\frac{1}{b_g}\sum [\grdfi (\curir) - \grdfi(\snap)] + \snapgrad - \left[\frac{1}{b_g} \sum \hesfi(\snap) - \snaphess \right](\curir - \snap) - \grdf (\curir)}^{3/2} \\
    &=\Ei \norm{\frac{1}{b_g} \sum [\fidiffnesres - (\fdiffnesres)]}^{3/2} \\
    &\leq \frac{1}{b_g^{3/4}}\left(\Ei \norm{\fidiffnesres - (\fdiffnesres)}^{2} \right)^{3/4} \\
    &\leq \frac{3^{3/4}}{b_g^{3/4}}\big(\Ei \norm{\fidiffnesres}^2 \\ & \quad \quad \,\,\, + \Ei\norm{(\fdiffnesres)}^{2} \big)^{3/4} \\
    &\leq \frac{3^{3/4}}{b_g^{3/4}}\left(\frac{L_2^2}{3}\norm{\curir - \snap}^4 + \frac{L_2^2}{3}\norm{\curir - \snap}^4 \right)^{3/4} \\
    &= \frac{2L_2^{3/2}}{b_g^{3/4}} \norm{\curir - \snap}^3.
\end{align*}
The first inequality is because of Lemma \ref{lemma:svrc31}. Indeed, as the different indices are independent, and the expectation is taken over the batch indices, we can apply Lemma~\ref{lemma:svrc31}. The second holds due to the basic inequality $\norm{\u+\v}^2 \leq 3(\norm{\u}^2 + \norm{\v}^2)$. The third inequality is because of Lemma~\ref{lemma:modnesterov}. 
\end{proof}

In the proof of Lemma~\ref{lemma:hessianestimator}, we will need the following matrix-moment inequality.
\begin{lemma}[Lemma 32 in \citet{zhou:svrc}]
\label{lemma:svrc32} Suppose that $q \geq 2, p \geq 2$, and fix $r \geq \max \{q,2\log p\}$. Consider i.i.d. random self-adjoint matrices $\mathbf{Y}_1,\ldots,\mathbf{Y}_N$ with dimension $p \times p$, $\E \mathbf{Y}_i = \mathbf{0}$. It holds that
$$
    \left[\E\norm{\sum_{i=1}^N \mathbf{Y}_i}^q\right]^{1/q} \leq 2 \sqrt{er}\norm{\left(\sum_{i=1}^N \E\mathbf{Y}_i^2\right)^{1/2}} + 4er(\E \max_i\norm{\mathbf{Y}_i}^q)^{1/q}.
$$
\end{lemma}

\begin{proof}{of Lemma \ref{lemma:hessianestimator}}
We can rewrite
\begin{align*}
    \Ej \hesdiffnorm^3 &= \Ej \norm{\hesf(\curir) - \frac{1}{b_h}\left[\sum[\hesfj(\curir) - \hesfj(\snap) + \snaphess]\right]}^3 \\
    &= \Ej \norm{\frac{1}{b_h}\left[\sum[\hesfj(\curir) - \hesfj(\snap) + \snaphess -  \hesf(\curir)]\right]}^3.
\end{align*}
Applying Lemma~\ref{lemma:svrc32}, and using our third-moment assumption, we can bound this further. The Lemma controls the third moment of a sum with a sum of second moments and an additive term of the third moment of the maximum matrix. While Assumption~\assthird{} is not ideal for bounding maximum terms, we may replace the maximum with a sum over the whole batch, which is sufficient in this case. This only makes the batch size requirement grow polylogarithmically in the dimension of the domain. We proceed with the proof. Define $\ycut = \hesfj(\curir) - \hesfj(\snap) + \snaphess -  \hesf(\curir)$ and set $N=b_h$, $q=3$, $p=d$ and $r=2\log p$. Then 
\begin{equation}
\label{eq:lemma32result}
    \left(\Ej \norm{\sum\ycut}^3\right)^{1/3} \leq 2\sqrt{er}\norm{\left(\sum \Ej \ycut^2\right)^{1/2}} + 4er(\Ej \max_{j_t} \norm{\ycut}^3)^{1/3}.
\end{equation}
We bound both terms separately. For the first, we follow the original proof and get
\begin{align*}
    2\sqrt{er}\norm{\left(\sum \Ej \ycut^2\right)^{1/2}} &= 2\sqrt{er}\norm{\sum \Ej \ycut^2}^{1/2} \\
    &= 2\sqrt{b_h er}\norm{\Ej \ycut^2}^{1/2} \\
    &\leq 2\sqrt{b_h er}\left(\Ej \norm{\ycut^2}\right)^{1/2} \\
    &\leq 2\sqrt{b_h er}\left(\Ej \norm{\ycut}^2\right)^{1/2}.
\end{align*}
Plugging back the definition of $\ycut$, and using Assumption~\assthird{} along with $\E[\abs{X}^s]^{1/s}\leq \E[\abs{X}^t]^{1/t}$ for $s \leq t$ allows us to bound
\begin{align}
    2\sqrt{b_h er}\left(\Ej \norm{\ycut}^2\right)^{1/2} &=  2\sqrt{b_h er}\left(\Ej \norm{\hesfj(\curir) - \hesfj(\snap) + \snaphess -  \hesf(\curir)}^2\right)^{1/2} \nonumber \\
    &\leq 2\sqrt{b_h er}\left(3\, \Ej \norm{\hesfj(\curir) - \hesfj(\snap)}^2 + 3\,\Ej\norm{\snaphess -  \hesf(\curir)}^2\right)^{1/2} \nonumber \\
    &\leq 2\sqrt{b_h er}\left(6\,L_2^2\xdiffnorm^2\right)^{1/2} \nonumber    \\
    &\leq 5L_2 \sqrt{b_h er}\xdiffnorm \label{eq:lemma32resultfirst}.
\end{align}
For the second term in Equation~\eqref{eq:lemma32result} we write
\begin{align}
    4er\big(\Ej \max_{j_t} \norm{\ycut}^3\big)^{1/3}
    &\leq 4er\big(\Ej \sum \norm{\ycut}^3\big)^{1/3} \nonumber \\
    &\leq 4b_h ^{1/3}er\big(\Ej  \norm{\ycut}^3\big)^{1/3} \nonumber \\
    &= 4b_h ^{1/3}er\big(\Ej \norm{\hesfj(\curir) - \hesfj(\snap) + \snaphess -  \hesf(\curir)}^3\big)^{1/3} \nonumber \\
    &\leq 4(7b_h )^{1/3}er\big(\Ej\norm{\hesfj(\curir) - \hesfj(\snap)} + \norm{\snaphess -  \hesf(\curir)}^3\big)^{1/3} \nonumber \\
    &\leq 4(7b_h )^{1/3}er\big(2 L_2^3\xdiffnorm^3\big)^{1/3} \nonumber\\
    &\leq 4(7b_h )^{1/3}er\big(2 L_2^3\xdiffnorm^3\big)^{1/3} \nonumber \\
    &\leq 10L_2b_h ^{1/3}er \xdiffnorm \label{eq:lemma32resultsecond}.
\end{align}
Plugging in Equations \eqref{eq:lemma32resultfirst} and \eqref{eq:lemma32resultsecond} into \eqref{eq:lemma32result} we get
\begin{align*}
    \left(\Ej \norm{\sum\ycut}^3\right)^{1/3} \leq 5L_2 \sqrt{b_h er}\xdiffnorm + 10L_2b_h ^{1/3}er \xdiffnorm,
\end{align*}
and therefore for the quantity we are interested in:
\begin{align}
    \Ej \hesdiffnorm^3 &\leq 125L_2^3 \left(\sqrt{\frac{er}{b_h }}+ \frac{2er}{b_h ^{2/3}}\right)^3 \xdiffnorm^3 \nonumber \\
    &\leq 125L_2^3 \left(\sqrt{\frac{2e\log d}{b_h }}+ \frac{4e\log d}{b_h ^{2/3}}\right)^3 \xdiffnorm^3 \label{eq:batch_size_relevant} \\
    &\leq 15000 L_2^3 \left({\frac{\log d}{b_h }}\right)^{3/2} \xdiffnorm^3. \nonumber
\end{align}
Because in \eqref{eq:batch_size_relevant} the first term in the parentheses dominates if $b_h \geq \sqrt{8e\log d}^6$, for which $b_h \geq 12000\log^3 d$ is sufficient.
\end{proof}
\subsection{Proof of Theorem \theothirdlow{}}
\label{section:meansquaredproof}
\begin{proof}{of Theorem \theothirdlow}
Let $\sigma, \lambda > 0$ be parameters yet to be chosen. The same is true for $d$ and $K$. According to Definition~\defrandhard, we define the scaled functions
$$
	f_i(\x) = {\sqrt[3]{n}\lambda \sigma^{3}} f^*_i\left(\frac{\x}{\sigma} \right) = \sqrt[3]{n} \lambda \sigma^{3} \hat{f}_{K;\B_i} \left(\frac{\C_i^T\x}{\sigma} \right),
$$
giving us
$$
	F(\x) = \frac{1}{n}\sum_{i=1}^n f_i(\x).
$$
We will choose the scaling parameters to ensure that our instance satisfies Assumption~\assthird{}, deriving the lower bound as we go along. We first guarantee smoothness:
for any $\x,\y \in \R^d$ we have
\begin{align}
	\E_i \norm{\nabla^2 f_i(\x) - \nabla^2 f_i(\y)}^3 \nonumber
	&= \frac{1}{n} \sum_{i=1}^n \norm{\nabla^2f_i(\x) - \nabla^2f_i(\y)}^3 \nonumber \\
	\label{eq:third_moment:lipschitztensor}
	&\leq \frac{1}{n} \sum_{i=1}^n (\sqrt[3]{n}\lambda \hat{\ell}_2)^3 \norm{\C_i^T \x - \C_i^T \y}^3 \\
	&=  \lambda^3 \hat{\ell}_2^3 \sum_{i=1}^n \norm{\C_i^T (\x -  \y)}^2 \norm{\C^T(\x-\y)} \nonumber\\
	&=  \lambda^3 \hat{\ell}_2^3 \norm{\C^T(\x -  \y)}^3 \nonumber\\
	&= \lambda^3 \hat{\ell}_2^3 \norm{\x -  \y}^3 \nonumber,
\end{align}
where \eqref{eq:third_moment:lipschitztensor} follows from Lemmas \ref{lemma:additional:tensorineq} and \ref{lemma:carmon:hat:properties}. So, the choice $\lambda = \frac{L_2}{\hat{\ell}_2}$ therefore accomplishes third-moment smoothness with parameter $L_2$. 

Now fix an algorithm $\algo$ and assume $\{[i^t,\x^{(t)}]\}_{t\in \N}$ are the iterates produced by $\algo$ on $F$. Consequently, by Lemma~\lemmainformed{} $\{[i^t,\x^{(t)}/\sigma]\}_{t\in \N}$ is informed by $F^*$. 
Therefore we can apply Lemma~\lemmagradbound{} on the sequence $\{[i^t,\x^{(t)}/\sigma]\}_{t\in \N}$ to get
\begin{align*}
\norm{\nabla F(\x^{(t)})}^2 
&= \norm{ \sqrt[3]{n} \lambda \sigma^{2} \nabla F^*(\x^{(t)}/\sigma)}^2  \\
&= n^{2/3} \lambda^2 \sigma^{4} \norm{ \nabla F^*(\x^{(t)}/\sigma)}^2  \\
&\geq n^{2/3}\lambda^2 \sigma^{4} \frac{1}{16n} \\
&= \frac{\sigma^{4}\lambda^2}{16n^{1/3}}.
\end{align*}
To get a lower bound for an $\varepsilon$ precision requirement we can choose $\sigma$ to be
$$
\frac{\sigma^{2}\lambda}{4 n^{1/6}}  = \epsilon \iff \sigma = \left(\frac{4\varepsilon \hat{\ell}_2 n^{1/6}}{L_2}\right)^{1/2}.
$$
Next, we will guarantee the optimality gap requirement.
We have
\begin{align*}
		F(\mathbf{0}) - \inf_{\x \in \R^d} F(\x)
		&\leq \sqrt[3]{n}\lambda \sigma^3 \left[ \frac{1}{n}\sum_{i=1}^n \hat{f}_{K;\B_i} \left(\frac{\C_i^T\mathbf{0}}{\sigma} \right) - \frac{1}{n}\sum_{i=1}^n \inf_{\x \in \R^d} \hat{f}_{K;\B_i} \left(\frac{\C_i^T\x}{\sigma} \right) \right]  \\
		&\leq \sqrt[3]{n}\lambda \sigma^3 \frac{1}{n}\sum_{i=1}^n\left[  \hat{f}_{K;\B_i} \left(\mathbf{0} \right) -  \inf_{\y \in \R^{d/n}} \hat{f}_{K;\B_i} \left(\y \right) \right]  \\
		&\leq 12\sqrt[3]{n}\lambda \sigma^3K,
\end{align*}
where the last step uses Lemma \ref{lemma:carmon:hat:properties} i).
We require 
$$
	12\sqrt[3]{n}\lambda \sigma^3K = 12\sqrt[3]{n} \frac{L_2}{\hat{\ell}_2} \left(\frac{4\varepsilon \hat{\ell}_2 n^{1/6}}{L_2}\right)^{3/2} K = 96n^{7/12}\left(\frac{\hat{\ell}_2}{L_2}\right) ^{1/2} \varepsilon^{3/2} K \leq \Delta.
$$
Our bounds get better with larger values of $K$, so we want to choose $K$ as
$$
	K = \left \lfloor  \frac{\Delta}{96 n^{7/12}} \left(\frac{L_2}{\hat{\ell}_2}\right)^{1/2} \frac{1}{\epsilon^{3/2}}\right \rfloor.
$$
We need $K\geq 1$ to have a sensible bound as becomes apparent below, and so we require
$$
	\tilde{c}\Delta L_2^{1/2}\frac{1}{\epsilon^{3/2}} \geq n^{7/12},
$$
or more concisely
$$
	 n \leq \frac{c\Delta^{12/7} L_2^{6/7}}{\epsilon^{18/7}},
$$
for some universal constants $c,\tilde{c}$.
As Lemma~\lemmagradbound{} yields the lower bound $T = \frac{nK}{2}$, we get a lower bound of
$$
	\Omega\left( \left(\frac{L_2}{\hat{\ell}_2}\right)^{1/2} \frac{\Delta n^{5/12}}{\epsilon^{3/2}} \right)
$$
with probability  at least $1/2$ for large enough dimension $d$ (see below). Thus there must be a fixed function $F$ such that for this many iterations -- with probability $1/2$ depending only on $\xi$ -- the iterates $\algo$ produces on $F$ all have gradient larger than $\varepsilon$. This means that 
$$ {T}_\epsilon (\algo, F) \geq \Omega\left( \frac{\sqrt{L_2}\Delta n^{5/12}}{\epsilon^{3/2}} \right).$$

For the requirement on the dimension $d$ for the bound from Lemma~\lemmagradbound{} to hold, we can plug in our values of $K$ and $\delta=1/2$ to see that some $d \in \tilde{\upp}(n^{2}\Delta L_2 \varepsilon^{-3})$ suffices. This concludes the proof.
\end{proof}

\section{Shared technical lemma}

We need the following result to guarantee the smoothness of our constructions.
\begin{lemma}
\label{lemma:additional:tensorineq}
Assume $m_1 \geq m_2$. Let $f : \R^{m_2} \rightarrow \R$ and for $\C \in \ortho(m_1,m_2)$ let $g : \R^{m_1} \rightarrow \R^{m_2}, \, \x \mapsto \C^T\x$. We will show that for any $\x,\y \in \R^{m_1}$:
\begin{equation*}
	\norm{\nabla^p[f(\C^T\x)] -\nabla^p[f(\C^T\y)]} 
	\leq \norm{\tilde{\nabla}^p f(\C^T\x) -\tilde{\nabla}^p f(\C^T\y)},
\end{equation*}
where the gradient operator $\nabla$ is with respect to $\x$ while $\tilde{\nabla}$ is with respect to $g(\x) = \C^T\x$. Further, if $f$ is $p$th-order smooth with constant $L$, then for any $\sigma > 0$
\begin{equation*}
	\norm{\nabla^p[\sigma^{p+1} f(\C^T\x/\sigma)] -\nabla^p[\sigma^{p+1}f(\C^T\y/\sigma )]} \leq L\norm{\C^T(\x - \y)}.
\end{equation*}
\end{lemma}

\begin{proof}{of Lemma \ref{lemma:additional:tensorineq}}
We are interested in the tensor $\nabla^p [f(\C^T\x)]$. 
Fix indices $i_1,\ldots,i_p$ and let $\Xi$ be the set of partitions of $[p]$. For a set $S \subset [p]$ let $i_S = \{i_j \,\vert \, j \in S \}$.  Define $\nabla^{\abs{S}}_{i_S}$ to be the order $\abs{S}$ partial derivative operator with respect to the coordinates with indices in $i_S$. Applying the higher-order chain rule we obtain
\begin{align*}
\nabla^p_{i_1,...,i_p}[f(\C^T\x)] = \sum_{(S_1,\ldots,S_L) \in \Xi}\sum_{j_1,\ldots,j_L = 1}^{m_2} \left( \prod_{l=1}^L \nabla_{i_{S_l}}^{\abs{S_l}} g_{j_l}(\x)\right) \tilde{\nabla}_{j_1,\ldots,j_L}^L f(\C^T\x).
\end{align*}
Now we use that $g_{j_l}$'s second and higher-order derivatives are zero, and that $\nabla_i g_{j_l}(\x) = \nabla_i [ \langle \cc_{j_l}, \x \rangle ] = c_{i,j_l}$. This means that in the above sum, the only partition that matters has $L=p$ and $\abs{S_1},\ldots, \abs{S_p} = 1$. W.l.o.g. we may take $S_l = \{l\}$ and consequently $i_{S_l} = \{i_l\}$. Then our expression simplifies to 
\begin{align*}
\nabla^p_{i_1,...,i_p}[f(\C^T\x)] &= \sum_{(S_1,\ldots,S_L) \in \Xi}\sum_{j_1,\ldots,j_L = 1}^{m_2} \left( \prod_{l=1}^L \nabla_{i_{S_l}}^{\abs{S_l}} g_{j_l}(\x)\right) \tilde{\nabla}_{j_1,\ldots,j_L}^L f(\C^T\x) \\
&= \sum_{j_1,\ldots,j_p = 1}^{m_2} \left( \prod_{l=1}^p \nabla_{i_{l}}g_{j_l}(\x)\right) \tilde{\nabla}_{j_1,\ldots,j_p}^p f(\C^T\x) \\
&= \sum_{j_1,\ldots,j_p = 1}^{m_2} \left( \prod_{l=1}^p c_{i_l,j_l}\right) \tilde{\nabla}_{j_1,\ldots,j_p}^p f(\C^T\x).
\end{align*}
We now bound the tensor operator norm from the Lemma statement: let $\vvv^{(1)},...,\vvv^{(p)} \in \R^{m_1}$ be arbitrary unit vectors. Then we have
\begin{align*}
&\left \langle \nabla^p[f(\C^T\x)] -\nabla^p[f(\C^T\y)], \, \vvv^{(1)} \otimes \cdots \otimes \vvv^{(p)} \right\rangle \\
= &\sum_{i_1,\ldots,i_p=1}^{m_1} v^{(1)}_{i_1} \cdots v^{(p)}_{i_p} \sum_{j_1,\ldots,j_p = 1}^{m_2} \left( \prod_{l=1}^p c_{i_l,j_l}\right) \tilde{\nabla}_{j_1,\ldots,j_p}^p (f(\C^T\x)-f(\C^T\y))\\
= &\sum_{j_1,\ldots,j_p = 1}^{m_2} \sum_{i_1,\ldots,i_p=1}^{m_1} v^{(1)}_{i_1} \cdots v^{(p)}_{i_p} \left( \prod_{l=1}^p c_{i_l,j_l}\right) \tilde{\nabla}_{j_1,\ldots,j_p}^p (f(\C^T\x)-f(\C^T\y))\\
= &\sum_{j_1,\ldots,j_p = 1}^{m_2} \sum_{i_1,\ldots,i_p=1}^{m_1} \left( \prod_{l=1}^p v^{(l)}_{i_l} c_{i_l,j_l}\right) 
\tilde{\nabla}_{j_1,\ldots,j_p}^p (f(\C^T\x)-f(\C^T\y))\\
= &\sum_{j_1,\ldots,j_p = 1}^{m_2} \left( \prod_{l=1}^p\left(  \sum_{i_l=1}^{m_1}  v^{(l)}_{i_l} c_{i_l,j_l}\right)\right)
\tilde{\nabla}_{j_1,\ldots,j_p}^p (f(\C^T\x)-f(\C^T\y))\\
= &\sum_{j_1,\ldots,j_p = 1}^{m_2} \left(\left(  \langle \vvv^{(1)}, \cc_{j_1} \rangle \right)\cdots \left(  \langle \vvv^{(p)}, \cc_{j_p} \rangle \right)\right) \tilde{\nabla}_{j_1,\ldots,j_p}^p (f(\C^T\x)-f(\C^T\y))\\
= &\sum_{j_1,\ldots,j_p = 1}^{m_2} \left( \left(\C^T \vvv^{(1)} \right)_{j_1}\cdots \left(\C^T \vvv^{(p)} \right)_{j_p} \right) 
\tilde{\nabla}_{j_1,\ldots,j_p}^p (f(\C^T\x)-f(\C^T\y)) \\
=& \left \langle \tilde{\nabla}^pf(\C^T\x)-\tilde{\nabla}^pf(\C^T\y), \, \C^T\vvv^{(1)} \otimes \cdots \otimes \C^T\vvv^{(p)} \right\rangle \\
\leq & \, \, \norm{\tilde{\nabla}^pf(\C^T\x)-\tilde{\nabla}^pf(\C^T\y)}.
\end{align*}
The first statement follows because $\C$ has orthonormal columns and can be extended to an $\R^{m_1\times m_1}$ matrix $\tilde{\C}$. Then $\norm{\C^T\vvv^{(k)}} \leq \norm{\tilde{\C}^T\vvv^{(k)} } = \norm{\vvv^{(k)}} = 1$ for all $k \in [p]$, which justifies the application of the operator norm definition. Because $\vvv^{(1)},...,\vvv^{(p)}$ were arbitrary, we obtain the desired inequality.

The second statement follows from $p$ applications of the chain rule.
\end{proof}

\end{document}